\documentclass[10pt,a4paper]{amsart}
\usepackage{latexsym,amssymb,amsmath,amsthm,amsfonts,enumerate,verbatim,xspace,
exscale}
\usepackage{graphicx}
\usepackage{color,amsbsy}
 
\usepackage{hyperref}
\definecolor{darkblue}{rgb}{0,0,.5}
\hypersetup{pdftex=true, colorlinks=true, breaklinks=true, linkcolor=darkblue, menucolor=darkblue, pagecolor=darkblue, urlcolor=darkblue,citecolor=darkblue}

\input xy 
\xyoption{all} 
\CompileMatrices
\UseComputerModernTips

\theoremstyle{plain}
\newtheorem{theorem}{Theorem}[section]
\newtheorem{lemma}[theorem]{Lemma}
\newtheorem{proposition}[theorem]{Proposition}

\theoremstyle{definition}
\newtheorem{definition}[theorem]{Definition}
\newtheorem{remark}[theorem]{Remark}

\def\D{\mathcal{D}}
\def\C{\mathcal{C}}
\def\R{\mathbb{R}}
\def\I{\mathbb{I}}

\def\M{\mathcal{M}}

\def\F{\mathcal{F}}

\def\A{\mathcal{A}}
\def\Ham{\mathcal{H}}

\newcommand{\caps}[1]{\textup{\textsc{#1}}}

\providecommand{\bysame}{\makebox[3em]{\hrulefill}\thinspace}

\newcommand{\up}{\upshape}

\newcommand{\leftmapsto}{\mbox{$\;\leftarrow\!\mapstochar\;$}}

\newcommand{\longto}{\longrightarrow}

\newcommand{\toto}{\twoheadrightarrow}
\def\vv<#1>{\langle#1\rangle}

\newcommand{\ev}{\mbox{$\text{\up{ev}}$}}

\providecommand{\det}{\mbox{$\text{\up{det}}\,$}}

\newcommand{\dd}[2]{\mbox{$\frac{\partial #2}{\partial #1}$}}

\providecommand{\del}{\partial}

\newcommand{\om}{\omega}
\newcommand{\Om}{\Omega}

\newcommand{\lam}{\lambda}

\newcommand{\wt}[1]{\mbox{$\widetilde{#1}$}}

\newcommand{\bsc}{\mbox{$\bigsqcup$}}

\newcommand{\by}[2]{\mbox{$\frac{#1}{#2}$}}
\newcommand{\cinf}{\mbox{$C^{\infty}$}}
\providecommand{\set}[1]{\mbox{$\{#1\}$}}

\newcommand{\fl}{\mbox{$\text{\up{Fl}}$}}

\newcommand{\X}{\mathfrak{X}}

\newcommand{\ver}{\mbox{$\textup{Ver}$}}
\newcommand{\hor}{\mbox{$\textup{Hor}$}}
\newcommand{\ann}{\mbox{$\textup{Ann}\,$}}

\newcommand{\spr}[1]{\mbox{$/\negmedspace/_{#1}$}}
\newcommand{\sporb}{\mbox{$/\negmedspace/_{\mathcal{O}}$}}
\newcommand{\momap}{momentum map\xspace}

\newcommand{\ao}{\mathfrak{a}}

\newcommand{\gu}{\mathfrak{g}}
\newcommand{\ho}{\mathfrak{h}}

\newcommand{\ko}{\mathfrak{k}}

\newcommand{\mo}{\mathfrak{m}}

\newcommand{\po}{\mathfrak{p}}
\newcommand{\qo}{\mathfrak{q}}

\newcommand{\Ad}{\mbox{$\text{\upshape{Ad}}$}}
\newcommand{\ad}{\mbox{$\text{\upshape{ad}}$}}

\newcommand{\orb}{\mbox{$\mathcal{O}$}}

\newcommand{\SU}{\mbox{$\textup{SU}$}}

\newcommand{\SO}{\mbox{$\textup{SO}$}}

\newcommand{\hl}{\mbox{$\textup{hl}$}}

\newcommand{\W}{\mbox{$\mathcal{W}$}}

\title[Symmetry reduction of Brownian motion
and
Quantum Calogero-Moser models]{%
Symmetry reduction of Brownian motion\\
and\\
Quantum Calogero-Moser models\\
}

\author{Simon Hochgerner}
\address{Section de Mathematiques,
Station 8,
EPFL, CH-1015 Lausanne}
\email{simon.hochgerner@epfl.ch} 
\begin{document}

\begin{abstract}
Let $Q$ be  a Riemannian $G$-manifold.  
This paper is concerned with the symmetry reduction of Brownian
motion in $Q$ and ramifications thereof in a Hamiltonian context. 
Specializing to the case of polar actions we
discuss
various versions of the stochastic Hamilton-Jacobi equation associated to
the symmetry reduction of Brownian motion  and observe 
some similarities to the 
Schr\"odinger equation of the quantum free particle reduction as
described by Feher and Pusztai~\cite{FP08}.
As an application we use this reduction scheme to derive examples of quantum Calogero-Moser
systems from a stochastic setting. 
\end{abstract}

\maketitle

\tableofcontents

\section*{Introduction}

Let $(Q,\mu)$ be a Riemannian manifold and $G$ a Lie group acting
properly and by isometries on $Q$. 
It is well-known that one can describe geodesic motion in $Q$ in terms
of a $G$-invariant  Hamiltonian system  $(T^*Q,\Om^Q,\Ham)$ where
$\Om^Q$ is the canonical symplectic form on $T^*Q$ and $\Ham$ is the
kinetic energy Hamiltonian. Because of the $G$-invariance this system may
be reduced 
by means of the standard \momap $J_G: T^*Q\to\gu^*$ 
at a given coadjoint orbit level $\orb\subset\gu^*$ to
yield the reduced Hamiltonian system
$(J_G^{-1}(\orb)/G,\Om_{\textup{red}},\Ham_{\textup{red}})$. The reduced
system may be a stratified system in the sense of \cite{SL91,OR04}. 
Solution curves in $J_G^{-1}(\orb)/G$ can be projected to curves in
$Q/G$ where they describe  the evolution of a mechanical system subject to a possibly
spin-dependent potential. Hence the reduced system is -in this case-
more complicated (and more interesting) than the original upstairs
system.
Similar remarks apply to the quantum version of this procedure:
Quantum Hamiltonian reduction of the free particle system will 
induce a new  (and generally non-free) quantum system on the reduced space $Q/G$.   
See \cite{KKS78,FP06,FP08,H04}. 

This paper is concerned with the stochastic version of this scheme. 
The idea is that the stochastic analogue of a free system is Brownian
motion. To further the analogy we consider the Hamiltonian description
(of \cite{LO08})
of Brownian motion in $Q$ and discuss its reduction with respect to
the symmetry group $G$. 
While this construction is very close to the classical one in its
approach we can also use the stochastic Hamilton-Jacobi equation to 
obtain the Schr\"odinger operator of the reduced quantum free particle
system as described by \cite{FP08}.
There are several versions of the stochastic Hamilton-Jacobi
equation. We comment first  on the one of \cite{GM83,N85} and then use
the formulation of \cite{LO09}. In order to obtain a reduced system
that more accurately reproduces the reduced quantum Hamiltonian
operator we also use a ``time-forward''-analogue of the stochastic
action of \cite{LO09}. 

Classical and quantum Calogero-Moser systems can be constructed
as projections of the free system 
when the configuration space is a semi-simple Lie group
or Lie algebra and the symmetry group is the group itself acting by
conjugation or adjoint action. 
See \cite{KKS78,FP06,FP08,H04}. 
We make a certain choice in this regard leading to rational versions
of these systems. 
Thus we can employ the above outlined
procedure to obtain a stochastic version of Calogero-Moser models and
pass, via the Hamilton-Jacobi equation associated to the
``time-forward''  stochastic action, to a stochastic representation of
the quantum Calogero-Moser Schr\"odinger equation as well as its
stationary solution wave functions.

\subsection*{Description of contents}
Section~\ref{sec:diff_in_mfd} starts by providing some general background on
diffusions on manifolds. 
Then we state a result of \cite{HR10} which  allows
for for symmetry reductions of diffusions defined in terms of a
Stratonovich operator which is equivariant in a certain sense with respect to a group
action.

In Section~\ref{sec:aHBM} we consider the general problem of
Hamiltonian construction of Brownian motion in $(Q,\mu)$ as well as
some reduction issues. 
This is mostly independent from Section~\ref{sec:BM-polar} but  
interesting in its own right. 
The construction of Brownian motion in $(Q,\mu)$ via the orthogonal
fame bundle $P$
is well understood, see \cite{IW89}: The idea is that one rolls the
manifold $Q$ along Brownian paths in $\R^n$ ($n=\dim Q$) without
slipping or twisting (\emph{rubber rolling}). The rolling is defined in terms of local
isometries between $Q$ and $\R^n$ thus involving the orthogonal frame
bundle. This gives rise to a Stratonovich equation 
on $P$ and its solution diffusion process $\Gamma^P$ projects to a
diffusion $\Gamma^Q$ in $Q$ which can be
shown to coincide with Brownian motion. Now the Hamiltonian version of
\cite{LO08} of this construction amounts to lifting the Stratonovich
equation on $P$ to a Stratonovich equation on $T^*P$ which is defined
in terms of Hamiltonian vector fields on $T^*P$ associated to
appropriate momentum functions. Thus we obtain a diffusion
$\Gamma^{T^*P}$ which projects to $\Gamma^P = \tau^P\circ\Gamma^{T^*P}$ via the foot point
projection $\tau^P: T^*P\to P$ and thus ultimately to Brownian motion
$\Gamma^Q$ in $Q$.

Now, in line with general experience in Hamiltonian mechanics, 
one would expect that there should also be a way to induce a
diffusion $\Gamma^{T^*Q}$ in $T^*Q = T^*P\spr{0}K$ from $\Gamma^{T^*P}$ such that
$\Gamma^Q = \tau^Q\circ\Gamma^{T^*Q}$ where $\tau^Q: T^*Q\to Q$. 
Here $T^*P\spr{0}K = J_K^{-1}(0)/K$ 
denotes the symplectic reduction of $T^*P$ at the
$0$-level set with respect to the standard \momap $J_K$ of the principal
$K$-action on $P$ cotangent lifted to $T^*Q$.
This is certainly possible if the manifold $Q$ is
parallelizable. However, in general $\Gamma^{T^*P}$ does not preserve
level sets of $J_K$. To overcome this deficency we redo the
Hamiltonian construction of $\Gamma^{T^*P}$ from a
\emph{non-holonomic} point of view. Thus we obtain a different
diffusion $\Gamma^{\mathcal{C}}$ which 
remains on $J_K^{-1}(0)$, projects to a diffusion $\Gamma^{T^*Q}$ on
$T^*Q = J_K^{-1}(0)/K$, and retains the basic feature $\Gamma^P =
\tau^P\circ\Gamma^{\mathcal{C}}$ whence also $\Gamma^Q =
\tau^Q\circ\Gamma^{T^*Q}$. It is maybe not surprising that this works
since the very idea of constructing Brownian motion via rubber rolling
is non-holonomic in its nature. 

We also make some comments on how  these Hamiltonian and non-holonomic (or almost
Hamiltonian) constructions behave in the presence of a symmetry group
$G$ acting properly and by isometries on $Q$.

Section~\ref{sec:BM-polar} is the main part of the paper. 
We assume that the $G$-action on $(Q,\mu)$ is actually
\emph{hyper-polar} which means that there exists an embedded
submanifold $M\subset Q$ which meets all $G$-orbits and does so
orthoganally and that $M$ is locally isometrically diffeomorphic to Euclidean
space $\R^l$.  
Moreover, it is assumed that the $G$-action  is of single orbit type. 
This ensures $Q$  to be paralellizable. 
Then we consider two different types of Hamiltonian constructions of
Brownian motion in $Q$.
Since the configuration space is parallelizable one can give a
Hamiltonian construction of Brownian motion in $Q$ by choosing a
global orthonormal basis and corresponding momentum fuctions on
$T^*Q$. 
The two choices which we consider are firstly that of a constant
frame (assuming it exists) and secondly that of a $G$-invariant frame
adapted to the decomposition into horizontal and vertical space (such
a frame can be constructed under the standing assumptions).
Thus we get two different Hamiltonian diffusions in $T^*Q$ both of
which project to Brownian motion in $Q$. 

Then we consider the symmetry reductions of these diffusions to
$Q/G$, $(T^*Q)/G$ and, where possible, to
$J_G^{-1}(\orb)/G\subset(T^*Q)/G$ where $J_G: T^*Q\to\gu^*$ is the
standard \momap of the cotangent lifted $G$-action on $T^*Q$ and
$\orb\subset\gu^*$ is a coadjoint orbit. 

At the $Q/G$-level we may consider the stochastic Hamilton-Jacobi
equation of Guerra-Morato~\cite{GM83} and Nelson~\cite{N85}. 
If the orbit $\orb$ is such that $J_G^{-1}(\orb)/G$ is diffeomorphic
to  a cotangent bundle (this happens e.g.\ with $\orb=0$), then we can
also invoke the stochastic Hamilton-Jacobi equation of Lazaro-Cami
and Ortega~\cite{LO09} associated to the
projected stochastic action $\wt{S}$ 
with respect to a choice of a (regular)
Lagrange $L_f$ submanifold in $J_G^{-1}(\orb)/G$. 
With 
\[
 \psi(t,x) := \delta^{-\frac{1}{2}}(x)E[\exp(-\widetilde{S}^x_t)],
\]
where $\delta$ is a function on $B$ depending on the inertia tensor
$\I$ associated to the metric $\mu$ 
we thus
find the following diffusion equation: 
\begin{equation}\tag{\ref{e:H1}}
 \frac{\del}{\del t}{}\psi
 = 
 \Big(\by{1}{2}\Delta 
 - \by{1}{2}\delta^{\frac{1}{2}}\Delta\delta^{-\frac{1}{2}} 
 + \by{1}{2}\vv<\lam,\I_x^{-1}(\lam)>\Big) \psi.
\end{equation}
This equation is analoguous to the Schr\"odinger equation of quantum
free partice reduction (with respect to polar actions)
described by
Feher and Pusztai~\cite[Thm.~4.5]{FP08}.

We treat the case $\orb=0$ separately in Section~\ref{sub:zero-orb}
and show that stationary solutions to \eqref{e:H1} are linked to
Lagrange submanifolds $L_f$ associated to eigenfunctions $f$ of the
radial part of the Laplace-Beltrami operator on $(Q,\mu)$. 
To obtain a diffusion equation which more accurately reproduces the
Schr\"odinger operator of the reduced quantum free particle system we
also introduce a ``time-forward'' formulation of the projected stochastic
action used by \cite{LO09} to obtain their version of the stochastic
Hamilton-Jacobi equation.

Finally in Section~\ref{sec:CMS} we apply the results of
Section~\ref{sec:BM-polar} to obtain stationary solutions to
(rational) quantum Calogero-Moser models associated to semi-simple Lie
algebras. This reproduces (and makes use of) some of the formulas
obtained by Olshanetsky and Perelomov~\cite{OP78,OP83}.

\emph{Acknowledgment.} I am grateful to the referees for their very
helpful remarks.

\section{Some stochastic geometry}\label{sec:diff_in_mfd}
This section begins with a review of some necessary definitions and results that
are all contained in the books \cite{IW89,E89}.
Then we state Theorem~\ref{thm:e-r} which
provides one of the approaches to  be used in subsequent
symmetry reduction schemes. 

\subsection{Diffusions on manifolds}
A diffusion is a continuous stochastic process which has the strong
Markov property. This is a concept which can be formulated in any
(decent) topological space.

Let $X$ be a locally compact topological space with one-point
compactification $\dot{X}=X\cup\set{\infty}$
and furnish $\dot{X}$ with its Borel $\sigma$-algebra $\mathcal{B}(X)$.
Define
$W(X)$ to be the set of all maps $w: [0,\infty)\to\dot{X}$ such that
there is a $\zeta(w)\in [0,\infty]$ satisfying
\begin{enumerate}[\up(1)]
\item
$w(t)\in X$ for all $t\in[0,\zeta(w))$ and $w: [0,\zeta(w))\to X$ is
    continuous;
\item
$w(t)=\infty$ for all $t\ge\zeta(w)$.
\end{enumerate}
Now $W(X)$ is equipped with the $\sigma$-algebra $\mathcal{B}(W(X))$ generated by all
Borel cylinder sets in $W(X)$. This $\sigma$-algebra has a natural
filtration given by the family of 
$(\mathcal{B}_t(W(X)))_{t\ge0}$ which are the $\sigma$-algebras generated by
Borel cylinder sets up to time $t$. 

A family of probabilities $(P_x)_{x\in\dot{X}}$ on
$(W(X),\mathcal{B}(W(X)))$ is said to be a
\emph{system of diffusion measures} on
$(W(X),\mathcal{B}(W(X)),\mathcal{B}_t(W(X)))$ if it has the strong
Markov property, for the definition of which we refer to
\cite[Section~IV.5]{IW89}.

Let $(\Om,\F,P)$ be a probability space and $\Gamma:
\Om\times\R_+\to\dot{X}$ a map.  
Define $\check{\Gamma}: \om\mapsto(t\mapsto\Gamma_t(\om))$. 
Then $\Gamma$ is said to be a (continuous) stochastic process in $X$ if
$\check{\Gamma}: (\Om,\F)\to(W(X),\mathcal{B}(W(X)))$ is a random
variable.\footnote{We consider only continuous processes.}
The law of $\Gamma$ is by definition the push-forward probability
$\check{\Gamma}_*P$ on $(W(X),\mathcal{B}(W(X)))$, i.e., $\check{\Gamma}_*P(S) =
P(\check{\Gamma}^{-1}(S))$ for all $S\in\mathcal{B}(W(X))$. 

The process
$\Gamma$ is a \emph{diffusion} in $X$ if there is a system
 of diffusion measures $(P_x)_{x\in\dot{X}}$ such that 
$\check{\Gamma}_*P = P_{\mu}$ as probability laws on
 $(W(X),\mathcal{B}(W(X)))$;
here 
\[
 P_{\mu}(S) = \int_{\dot{X}}P_x(S)\mu(dx)\;
 \textup{ for all }
 S\in\mathcal{B}(W(X))
\]
and $\mu = (\Gamma_0)_*P: \mathcal{B}(X)\to[0,1]$ is the initial
distribution of $\Gamma$. 
In practice $P_x$ will be obtained by push-forward of $P$ with respect
to the stochastic process $\Gamma^x$ which is conditioned such that
$\Gamma^x_0 = x$ a.s.

\subsubsection{Diffusions via Stratonovich equations}
Let now $X=Q$ be a manifold. If $N$ is another manifold then a
\emph{Stratonovich operator} $\mathcal{S}$ from $TN$ to $TQ$ is a
section of $T^*N\otimes TQ\to N\times Q$. 
Equivalently we can view $\mathcal{S}$ as a smooth map 
$\mathcal{S}: Q\times TN\to TQ$ 
which is linear in the fibers and sits over the identity on
$Q$. 
Let $X_0,X_1,\dots,X_k$ be
vector fields on $Q$ and define the associated Stratonovich operator 
$\mathcal{S}$
from $T\R^{k+1}$ to $TQ$ by 
\[
 \mathcal{S}: Q\times T\R^{k+1}\longto TQ,\;
 (x,w,w')\longmapsto\sum_{i=0}^k X_i(x)\vv<e_i,w'>
\]
where $e_i$ denotes the standard basis in $\R^{k+1}$. 
We remark that the number $k$ is not related to the dimension of $Q$. 
Assume $(\F_t)$ is an increasing filtration of $\F$ which is
right-continuous. 
It is then a reference family in the sense of \cite[p.~20]{IW89}.
Consider the process $Y: \Om\times\R_+\to\R^{k+1}$,
$(t,\om)\mapsto(t,W_t(\om))$, where $W$ denotes a continuous version of
$(\F_t)$-adapted Brownian motion in $\R^{k}$. 
We will be concerned with Stratonovich equations of the form 
\begin{equation}\label{e:S1}
 \delta\Gamma = \mathcal{S}(Y,\Gamma)\delta Y.
\end{equation}
A continuous $(\F_t)$-adapted process $\Gamma: \Om\times\R_+\to Q$  is
called a solution to 
\eqref{e:S1} if there is a continuous version $W = (W^i)$ of
$(\F_t)$-adapted Brownian motion in $\R^{k}$ such that, in the
Stratonovich sense,
\begin{equation}\label{e:def-of-sol}
 f(\Gamma_t) - f(\Gamma_0)
 = \int_0^t(X_0f)(\Gamma_s)ds
   + \sum_{i=1}^k\int_0^t(X_if)(\Gamma_s)\delta W_s^i
\end{equation}
for all smooth functions $f$ of compact support on $Q$.

Suppose $\Gamma$ is a solution to \eqref{e:S1} such that $\Gamma_0=x$
a.s.\ and $\Gamma$ satisfies \eqref{e:def-of-sol} with respect to a
version $W$. Then we will write $\Gamma=\Gamma^{x,W}$ to remember
these data.
The following is an account of Theorems~V.1.1 and  V.1.2 in \cite{IW89}

\begin{theorem}
Let the assumptions be as above and consider Equation~\eqref{e:S1}.
\begin{enumerate}[\up (1)]
\item
For each initial condition, $\Gamma_0=x$ a.s.,  
and continuous $(\F_t)$-adapted Brownian motion $W$,
a solution $\Gamma^{x,W}$ exists and is unique up to explosion time. 
\item
Let $P_x := \check{\Gamma}^{x,W}_*P$. Then $P_x$ is independent of $W$
and $(P_x)$ is a system of diffusion measures generated by the second
order differential operator
\begin{equation}\label{e:A-thm}
 A = X_0 + \by{1}{2}\sum_{i=1}^k X_iX_i.
\end{equation}
which acts on the space of smooth functions with compact support 
$\cinf(Q)_0$. 
\end{enumerate}
\end{theorem}

In general, if $\Gamma$ is a diffusion in $Q$ such that the associated
system of diffusion measures is unique and is generated by a second
order differential operator
$A$, then $A$ is also called the generator of $\Gamma$ and $\Gamma$ is
said to be an $A$-diffusion.
This does not require $Q$ to be a manifold; if $Q$ is a topological
space then a generator is a linear operator $A$ on the Banach space of
continuous functions $C(\dot{Q})$ with domain of definition $\D(A)$.
See \cite[Section~IV.5]{IW89}.

Assume that $Q$ is endowed with a linear connection $\nabla: TQ\times
TQ\to TQ$. For vector fields $X,Y\in\X(Q)$ the Hessian of
$f\in\cinf(Q)$ is
$\textup{Hess}^{\nabla}(f)(X,Y) = XY(f) - \nabla_XY(f)$. This is
bilinear in $X$ and $Y$ but not symmetric, unless $\nabla$ is
torsion-free. 

\begin{definition}[Drift, Martingale, Brownian motion]\label{def:drift}
$\phantom{t}$
\begin{itemize}
\item
Let $\Gamma$ be a diffusion in $Q$ with generator $A$. 
Then the \emph{drift} of
$\Gamma$ with respect to $\nabla$ is defined to be the first order
part of $A$, which is determined by $\nabla$. If $A$ is of the form
\eqref{e:A-thm} then this is $X_0+\by{1}{2}\sum\nabla_{X_i}X_i$. 
\item
According to \cite[Chapter~IV]{E89} the $A$-diffusion $\Gamma$ is a
\emph{martingale} in $(Q,\nabla)$ if 
$A$
is purely second order with
respect to $\nabla$, i.e., the $\nabla$-drift vanishes. 
In \cite{E89} this is stated for torsion-free connections but it is
noted that one can use the same definition for connections with
torsion. 
\item
If $(Q,\mu)$ is a Riemannian manifold then an $A$-diffusion
is called \emph{Brownian motion} if $A=\by{1}{2}\Delta$ where $\Delta$ is the
metric Laplacian. 
\end{itemize}
\end{definition}

%%%%%%%%%%%%%%%%%%%%%%%%%%%%%%%%%%%%%%%%%%%%%%%%%%%%%%%%%%%%%%%%%%%%%%%%%%%%%%%%%%%%%%
%%%%%%%%justif....
%%%%%%%%%%%%%%%%%%%%%%%%%%%%%%%%%%%%%%%%%%%%%%%%%%%%%%%%%%%%%%%%%%%%%%%%%%%%%%%%%%%%%%%

As it stands, Brownian motion is not unique. There may be several
diffusions in $Q$ such that the associated system of diffusion
measures is generated by $A=\by{1}{2}\Delta$. However, since we have
defined diffusions in terms of systems of diffusion measures we
regard two diffusions which give rise to the same system of diffusion
measures as equivalent. Now the system of diffusion measures generated
by $A=\by{1}{2}\Delta$ is unique and it is in this sense that we think
of Brownian motion as being unique.

To construct Brownian motion in $(Q,\mu)$ consider the orthonormal
frame bundle $\rho: \F\to Q$ over $(Q,\mu)$. 
The Levi-Civita connection on $Q$ gives rise to a unique principal
bundle connection $\om$ on $\rho: \F\to Q$.
An element $u\in\F$ can
be regarded as an isometry $u: \R^d\to T_{\rho(u)}Q$
where $d=\dim Q$. Let $(e_i)$ be the standard basis in $\R^d$. Define
the canonical horizontal vector fields $L_i\in\X(\F,\hor^{\om})$,
$i=1,\dots,d$,
by
\begin{equation}\label{e:L}
 L_i(u) = \textup{hl}^{\om}_u(u(e_i))
\end{equation}
where $\textup{hl}^{\om}: \X(Q)\to\X(\F)$ is the horizontal lift map
of $\om$.
If $(W^i)$ is Brownian motion in $\R^d$ and $\Gamma$ solves the
Stratonovich equation
\begin{equation}\label{e:SL}
 \delta\Gamma = \sum L_i(\Gamma)\delta W^i
\end{equation}
then $\rho\circ\Gamma$ is a diffusion in $Q$ with generator
$\by{1}{2}\Delta^{\mu}$, that is, a Brownian motion. This is explained
in \cite[Chapter~V.4]{IW89} and follows also from
Theorem~\ref{thm:e-r} below; the essential observation in
this context is that the Stratonovich operator $\mathcal{S}$ of \eqref{e:SL} enjoys
the equivariance relation 
\[
 \mathcal{S}(ug,g^{-1}w,g^{-1}w') 
 = \big(\textup{hl}^{\om}(u)(u(ge_i))\vv<e_i,g^{-1}e_i>\big)g 
 = \mathcal{S}(u,w,w')g
\] 
for the principal right action of the
structure group.
To connect with Theorem~\ref{thm:e-r} the principal right action
can be turned to a left action via inversion in the group.

\subsection{Equivariant reduction}\label{sec:equiv-red}

Let $(\Om,\F,(\F_t),P)$, $Q$, $X_0,X_1,\dots,X_k\in\X(Q)$ and
$\delta\Gamma = \mathcal{S}(Y,\Gamma)\delta Y$ as before. Suppose
there is a Lie group $G$ which acts properly on $Q$ from the left. 
We extend this action to $\dot{Q}$ by requiring $\infty$ to be
a fixed point. Let $\pi: Q\toto Q/G$ be the projection
and $\cinf(Q)^G$ denote the subspace of $G$-invariant smooth functions
on $Q$.
Note that $Q/G$ need not be a manifold; in general $Q/G$ is a
topological space which is naturally stratified by smooth components.

In the following all actions
are tangent lifted where appropriate without further notice.

The proof of the following theorem is sketched in the appendix. 

\begin{theorem}[Eqivariant reduction~\cite{HR10}]\label{thm:e-r}
Suppose there is a group representation $\rho: G\to\textup{O}(k)$ 
and let $\textup{O}(k)$ act on $\R^{k+1}=\R\times\R^k$ such that the first
factor is acted upon trivially.
If
$\mathcal{S}$ satisfies the equivariance property
\begin{equation}%\label{e:S-equiv}
 \mathcal{S}(gx,\rho(g)y,\rho(g)y')
 = g.\mathcal{S}(x,y,y')
\end{equation}
for all $(x,y,y')\in Q\times T\R^{k+1}$ then the diffusion $\Gamma$
induces a diffusion
$\pi\circ\Gamma$ in $Q/G$. Moreover, $A$ preserves $\cinf(Q)^G$, and
the induced generator $A_0$ of $\pi\circ\Gamma$ is
characterized by 
\begin{equation}\label{e:A_0}
 \pi^*A_0f = A\pi^*f
\end{equation}
for $f\in\cinf(Q/G) := \set{f\in C(Q/G): \pi^*f\in\cinf(Q)^G}$.
\end{theorem}

Equivariant reduction is a natural extension of the reduction theory
of \cite[Theorem~3.1]{LO08a}. While the results of \cite{LO08a} are
stronger in the sense that they provide a Stratonovich equation on the
base space $Q/G$, they are only applicable when the original Stratonovich
operator is $G$-invariant. This means that  
\begin{equation}
 \mathcal{S}(gx,y,y')
 = g.\mathcal{S}(x,y,y')
\end{equation}
for all $g\in G$ and $x\in Q$ and \cite{LO08a} show how to obtain an 
induced Stratonovich operator on the base space $Q/G$.
This covers the case where the
Stratonovich operator is defined in terms of $G$-invariant vector fields.
To connect with the above theorem, note that a $G$-invariant operator
$\mathcal{S}$ can be considered $G$-equivariant with respect to the
trivial action on the source space $\R^{k+1}$, i.e., where
$\rho(g)y=y$ for all $g\in G$ and $y\in\R^{k+1}$. 
By contrast the observation in
equivariant reduction is that although the upstairs Stratonovich
operator $\mathcal{S}$ is not projectable to $Q/G$, the diffusion still factors to a diffusion
in the base space and the generator of the downstairs diffusion is induced from that of the
original diffusion on $Q$.

\section{The (almost) Hamiltonian construction of Brownian motion and
  symmetry reduction}\label{sec:aHBM}

\subsection{General version}
Let $(Q,\mu)$ be a Riemannian manifold, 
$G$ a Lie group acting properly and through isometries on $Q$. 
We do not require the action to be free.
Assume $\Ham: T^*Q\to\R$ is a $G$-invariant Hamiltonian
function (e.g., the kinetic energy associated to $\mu$).
We want to describe a perturbation  by a Brownian motion
of the Hamiltonian 
system $(T^*Q,\Om^Q,\Ham)$ with dynamics 
$X_{\mathcal{H}}$;
here  $X_{\mathcal{H}}$ is the Hamiltonian vector field with respect to
the canonical exact symplectic form $\Om^Q$ on $T^*Q$.
If $Q = \R^n$ was a vector space and $\mu$ the Euclidean metric then 
this should give rise to an Ito equation
of the form 
\begin{equation}\label{e:HD-0}
 d\Gamma = X_{\mathcal{H}}(\Gamma)dt + dW
\end{equation} 
with $W$ 
Brownian motion  in $Q$. 
This is Bismut's notion of a Hamiltonian diffusion described in
\cite{Bis81}.  
Lazaro-Cami and Ortega~\cite{LO08} have generalized
this set-up to non-Euclidean spaces and incorporated general
semi-martingales as source of the noise.

To generalize \eqref{e:HD-0} to the non-Euclidean setting one notices
that it defines a diffusion with generator 
$X_{\mathcal{H}} + \by{1}{2}\Delta^{\mu}$ where $\Delta^{\mu} =
\sum\del_i\del_i$ is the
Laplace-Beltrami operator on
$(Q,\mu)$ viewed as acting on $\cinf(T^*Q)$, 
thus identifying $\del_i = \dd{q^i}{}$ with $(\del_i,0) =
\sum\dd{p_i}{H^i}\dd{q^i}{}$ which is the Hamiltonian vector field of
$H^i: T^*Q\to\R$, $(q,p)\mapsto\vv<p,e_i>$ with $\dd{q^i}{}=e_i$ the
standard basis. 
Below we will give a general construction of such a
diffusion. 
In doing so we follow essentially the ideas of \cite{LO08}.

Thus we
consider the orthonormal frame bundle $\rho: P\to Q$ and
denote its structure group by $K=O(n)$, $n=\dim(Q)$. 
The Levi-Civita connection
$\nabla$ on $(Q,\mu)$ induces a unique principal bundle connection
form $\om: TP\to\ko$. 
We use $\om$ to decompose $TQ=\hor^{\om}\oplus\ver^{\rho}$ into
horizontal space $\hor^{\om}=\ker\om$ and vertical space
$\ver^{\rho}=\ker T\rho$.
The soldering form $\sigma\in\Om^1(P,\R^n)$ is
defined by $\sigma(\xi_u) = u^{-1}(T_u\rho.\xi)$ for $\xi_u\in T_uP$,
and this induces a trivialization of the horizontal bundle.
Therefore, we obtain a trivialization of $TP$ by
\begin{equation}\label{e:triv-P}
 \Phi: TP\longto P\times \R^n\times\ko, 
 \quad
 \xi_u\longmapsto\Big(u,\sigma(\xi_u),\om(\xi_u)\Big)
\end{equation}
We may use this isomorphism to obtain a metric $\mu^P$ on $P$ by using
the standard inner products on $\R^n$ and $\ko$ and requiring $\Phi$ to
be an isometry. 
If $e_1,\dots,e_n$ is the standard basis in $\R^n$ then we may use
$\Phi$ to define vector fields $L_i\in\X(P,\hor^{\om})$ by
\[
 L_i(u) = \Phi^{-1}(u,e_i,0) = \hl^{\om}_u(u(e_i))
\]
where $\hl^{\om}: P\times_Q TQ\to \hor^{\om}$ is the horizontal lift mapping.
The $L_i$ are called the canonical horizontal vector fields.
Notice that they form a global orthonormal frame for $\hor^{\om}$.

We remark that $\mu^P$ is $K$-invariant and the $K$-invariant
subspaces $\hor^{\om}$ and $\ver^{\rho}$ are perpendicular with
respect to $\mu^P$. 
By construction $\om$ coincides with the
mechanical connection (defined in \eqref{e:MC}) on $(P,\mu^P)$. 

Furthermore, since $G$ acts by isometries on $Q$ it induces an action
on $P$ which also preserves $\mu^P$: an element $u\in P$ is viewed as
an isometry $u: \R^n\to T_{\rho(u)}Q$ and $g\in G$ acts from the left
via $g.u = T_{\rho(u)}g\circ u: \R^n\to T_{g\rho(u)}Q$. 

Now the principal action of $K$ on $P$ is a right action and to
convert it to a left action we use inversion in the group. 
The actions of $K$ and $G$
commute whence the product $K\times G$ also acts on $P$. 
The cotangent lifted action by $K\times G$ is Hamiltonian. 
Singular Reduction in Stages thus implies that
\[
 T^*P\spr{0\times\mathcal{O}}(K\times G)
 = (T^*P\spr{0}K)\sporb G
 = T^*Q \sporb G
\]
as stratified symplectic spaces;
here $\sporb$ denotes symplectic reduction at the coadjoint orbit
$\orb\subset\gu^*$ with respect to the obvious momentum map. 
See \cite{MPROM07}.
Let $\mathcal{H}^P := \rho^*\Ham$ 
be the $K\times G$-invariant 
upstairs Hamiltonian where $\rho: T^*P\to T^*Q$ is defined in
terms of the metric isomorphisms $T^*P\cong_{\mu^P}TP$ and
$T^*Q\cong_{\mu^Q}TQ$. 
Then
(regular)
Hamiltonian reduction of $(T^*P,\Om^P,\Ham)$ at $0\in\ko^*$ with
respect to $K$ yields the original Hamiltonian system
$(T^*Q,\Om^Q,\Ham)$. 
Consequently (singular) Hamiltonian reduction at
$0\times\orb\subset\ko^*\times\gu^*$ of $(T^*P,\Om^P,\Ham^P)$ is
equivalent to (singular) 
Hamiltonian reduction of $(T^*Q,\Om^Q,\Ham)$ at
$\orb\subset\gu^*$.  

Let us define the momentum functions 
\[
 H^i: T^*P\longto\R,
 \quad
 \eta_u\longmapsto\vv<\eta_u,L_i(u)>
\]
and the associated Hamiltonian vector fields $X_{H^i}$. Note that
$T\tau_P.X_{H^i}(\eta_u) = L_i(u)$ where $\tau_P: T^*P\to P$ is the
tangent projection. Therefore, if $\Gamma^{T^*P}$
satisfies $\delta\Gamma^{T^*P} = \sum X_{H^i}(\Gamma^{T^*P})\delta
W^i$ the projection $\tau_P\circ\Gamma^{T^*P}$ satisfies \eqref{e:SL},
and $\tau_P\circ\Gamma^{T^*P}$ is Brownian motion in $(Q,\mu)$. This
is the Hamiltonian construction of Brownian motion of \cite{LO08}.

Consider the Stratonovich operator
\begin{align*}
 \mathcal{S}^{T^*P}: T^*P\times T(\R\times\R^n)&\longto T(T^*P),\quad
 (\eta_u,t,t',w,w')
 \longmapsto
 X_{\rho^*\mathcal{H}}(\eta_u)t'
 + \sum X_{H^i}(\eta_u)\vv<e_i,w'>.
\end{align*}
Let $K$ act on $\R\times\R^n$ by acting on the second factor only and
$G$ act trivially on $R\times\R^n$. Thus $K\times G$ acts on
$\R\times\R^n$. 

Let $J_K: T^*P\to\ko^*$ denote the standard \momap of the cotangent
lifted $K$-action on $T^*P$. Then 
\[
 J_K^{-1}(0) = \set{\lambda\in T^*P: 
 \vv<\lambda,X>=0\textup{ for all } X\in\ver^{\rho}}
 = \hor^*
\]
 where $\hor^*$ is
called the dual horizontal bundle. Contrary to $\hor^{\om}$ the dual
horizontal $\hor^*$ is defined connection independently. 
Further, let $J_G: T^*P\to\gu^*$ denote the \momap of the lifted
$G$-action on $T^*P$, and correspondingly $J_{K\times G} = (J_K,J_G):
T^*P\to\ko^*\times\gu^*$ for the product action of $K\times G$. 

\begin{lemma}\label{lm:equiv-P}
The following hold.
\begin{enumerate}[\up (1)]
\item
The operator $\mathcal{S}^{T^*P}$ is $K\times G$-equivariant. 
\item
If $\orb\subset\gu^*$ is a codjoint orbit, then
$\mathcal{S}^{T^*P}$ restricts to a Stratonovich operator on
$J_G^{-1}(\mathcal{O})$ in the
sense that
\[
 \mathcal{S}^{T^*P}:
 J_{G}^{-1}(\orb)\times T(\R\times\R^n)
 \longto T\Big(J_{G}^{-1}(\orb)\Big).
\]
If the $G$-action is non-free then the last statement holds for each
smooth stratum in $J_{G}^{-1}(\orb)$.
\end{enumerate}
\end{lemma}

\begin{proof}
(1)
Clearly $X_{\rho^*\mathcal{H}}$ is $K\times G$-invariant. Furthermore,
we have, for $g\in G$, 
\[
 H^i(g.\eta_u)
 = \vv<g.\eta_u, L_i(g.u)> 
 = \vv<g.\eta_u,g.L_i(u)>
 = \vv<\eta_u,L_i(u)>
 = H^i(\eta_u)
\]
whence $H^i$ and thus also $X_{H^i}$ are $G$-invariant. This
establishes $G$-equivariance of $\mathcal{S}^{T^*P}$. 

To see $K$-equivariance we start by defining 
\[
 H_{w'}: T^*P\longto\R,
 \quad
 \eta_u\longmapsto\sum H^i(\eta_u)\vv<e_i,w'>
 = \sum\vv<\eta_u,\hl^{\om}_u\Big(u(w')\Big)>
\]
for $w'\in\R^n$. We note that 
\[
 H_{k.w'}(k.\eta_u)  
 = \sum\vv<k.\eta_u,\hl_{k.u}^{\om}\Big(u(\vv<e_i,k.w'>k^{-1}e_i)\Big)>
 = \sum\vv<k.\eta_u,k.\hl^{\om}_{u}\Big(u(w')\Big)>
 = H_{w'}(\eta_u)  
\]
Therefore, if we fix a vector $Y_{\eta_u}\in T_{\eta_u}(T^*P)$ and
choose a (local) $K$-invariant vector field  $Y\in\X(T^*P)$ such that
$Y(\eta_u) = Y_{\eta_u}$, we find 
\begin{align*}
 d\Big(\sum H^i\vv<e_i,k.w'>\Big)(k.\eta_u).T_{\eta_u}k.Y_{\eta_u}
 &= \dd{t}{}|_0\Big((\fl_t^Y)^*H_{k.w'}\Big)(k.\eta_u)
 = \dd{t}{}|_0 H_{k.w'}(\fl_t^Y(k.\eta_u))\\
 &= \dd{t}{}|_0 H_{k.w'}(k.\fl_t^Y(\eta_u))
 = \dd{t}{}|_0 H_{w'}(\fl_t^Y(\eta_u))\\
 &= d\Big(\sum H^i\vv<e_i,w'>\Big)(\eta_u).Y_{\eta_u}
\end{align*}
Putting this together yields
\begin{align*}
 \sum X_{H^i}(k.\eta_u)\vv<e_i,k.w'>
 &= 
 (\Om^P)^{-1}_{k.\eta_u}
 \Big(d\Big(\sum H^i\vv<e_i,k.w'>\Big)(k.\eta_u)\Big)\\
 &= 
 T_{\eta_u}k.(\Om^P)^{-1}_{\eta_u}
 \Big(d\Big(\sum H^i\vv<e_i,k.w'>\Big)(k.\eta_u)\circ T_{\eta_u}k\Big)\\
 &= 
 T_{\eta_u}k.(\Om^P)^{-1}_{\eta_u}
 \Big(d\Big(\sum H^i\vv<e_i,w'>\Big)(\eta_u) \Big)\\
 &=
 T_{\eta_u}k.\sum X_{H^i}(\eta_u)\vv<e_i,w'>
\end{align*}
whence $K$-equivariance of $\mathcal{S}^{T^*P}$ also follows.

(2).
Since the $H^i$ are $G$-invariant 
Noether's Theorem implies that the
Hamiltonian vector fields used in the
definition of $\mathcal{S}^{T^*P}$ are tangent to $J_G^{-1}(\orb)$. 
This is also true in the stratified context and moreover  strata are preserved
by Hamiltonian vector fields of invariant functions. See \cite{SL91,OR04}. 
\end{proof}

Now to reduce the Hamiltonian diffusion to $T^*Q\sporb G$ we should
check that $\mathcal{S}^{T^*P}$ restricts also to a Stratonovich
operator on
$J_K^{-1}(0) = \hor^*\subset T^*P$. For then, by equivariant
reduction, we would obtain a diffusion in $(J_K^{-1}(0)\cap
J_G^{-1}(\orb))/(K\times G) = J_{K\times G}^{-1}(0\times\orb)/(K\times G) =
T^*Q\sporb G$. Thus for $\eta_u\in\hor^*$ we should show that
$\mathcal{S}^{T^*P}(\eta_u,w,w')\in T_{\eta_u}(J_K^{-1}(0))=\ker
dJ_K(\eta_u)$; that is, for $Y\in\ko$ and with $\zeta^{K,T^*P}_Y$
denoting the fundamental vector field associated to the cotangent
lifted $K$-action on $T^*P$,
\begin{align*} 
 dJ_K.\sum X_{H^i}(\eta_u)\vv<e_i,w'>
 &= -\sum dH^i(\eta_u).\zeta_Y^{K,T^*P}(\eta_u)\vv<e_i,w'>
 = -\dd{t}{}|_0H_{w'}(\exp(tY).\eta_u)\vv<e_i,w'>\\
 &= -\dd{t}{}|_0H_{w'}(\eta_u)\vv<e_i,\exp(-tY).w'>
 = H_{w'}(\eta_u)\vv<e_i,Y.w'> 
 = \vv<\eta_u,\hl^{\om}_u(w')>.
\end{align*}
However, this expression is certainly not $0$ for general $\eta_u$, $Y$ and $w'$.

Therefore, while $\rho\circ\tau\circ\Gamma^{T^*P}$ yields Brownian
motion in $(Q,\mu)$ (see above), we 
\emph{do not obtain an induced diffusion} in $T^*Q = \hor^*/K$. 
This problem can be resolved 
reflecting on the nature of the
construction of Brownian motion in the Riemannian manifold $Q$:
A Brownian path in $Q$ is constructing by rolling $Q$ along a Brownian
path in $\R^n$ \emph{without slipping or twisting}, i.e., rubber
rolling on a rough surface. 
Thus we are dealing with a \emph{non-holonomic} control problem. 
The configuration space of this problem is 
\[
 \wt{Q} 
 := \set{(x,w,u): x\in Q, w\in\R^n, u: T_xQ\to\R^n\textup{ is an
     isometry}}
 \cong P\times\R^n
\]
and the constraint distribution which specifies the set of allowed
motions is 
\[
 \wt{\D}
 :=
 \set{(X_u,w,w')\in T(P\times\R^n): w'=\sigma(X_u)\And
   X_u\in\hor_u^{\om}}
 \cong\hor^{\om}\times\R^n
\]
where $\sigma\in\Om^{1}(P)$ is the soldering form as above. 
The abelian group $\R^n$ acts on the pair $(\wt{Q},\wt{D})$ by
so-called outer symmetries which reflects the fact that the
constraints are the same at all points of the ``surface'' $\R^n$ along
which the rolling takes place. Thus we can reduce by the abelian
$\R^n$ action and obtain a new configuration space - distribution pair
$(P,\D:=\hor^{\om})$. Suppose $\xi_i(t): \Om\to\R$ are a set of
time-dependent controls. (This corresponds to the Brownian noise.)
Then the non-holonomic control problem
\begin{equation}\label{e:control-problem}
 u'(t) = \sum L_i(u(t))\xi_i(t)
\end{equation}
has an almost Hamiltonian formulation which is given by the ODE on
$\M := \mu(\D) = \mu(\hor^{\om}) = \hor^* = J_K^{-1}(0)$ 
\[
 \gamma'(t)
 = \sum \mathcal{P}(\gamma(t))X_{H^i}(\gamma(t))\xi_i(t)
\]
(see \cite{C02,B03}) 
where $\mathcal{P}: T(T^*P)|\M\to\C\subset T\M$ 
is the Hamiltonian encoding of the \emph{constraint force
projection operator} \eqref{e:PTMC}.\footnote{The notation $\M$ and $\C$ is
  introduced in order to be consistent with the non-holonomic
  literature such as \cite{BS93}.}
With $\tau: T^*P\to P$ being the footpoint projection 
the space $\C$ is defined as 
\[
 \C 
 := 
 \set{Z\in T\M: T\tau.Z \in \D_{\tau(Z)} = \D\cap T_{\tau(Z)}P }.
\]
Now it can be shown that $T(T^*P)|\M = \C\oplus\C^{\Om}$ where
$\C^{\Om}$ is the $\Om$-orthogonal of $\C$ in $T(T^*P)|\M$ with
respect to the canonical symplectic form $\Om$ on $T^*P$.
Then 
\begin{equation}\label{e:PTMC}
 \mathcal{P}: T(T^*P)|\M\longto\C\subset T\M
\end{equation}  
is defined as the projection
along $\C^{\Om}$. 
See \cite{BS93}. 
Let $\Pi: TP\to\hor^{\om}$ be the projection along the vertical
space. It follows that $T\tau\circ \mathcal{P} = \Pi\circ\tau$. (See
\cite[Section~2]{HR10}). Since $\Pi.L_i = L_i$ 
and $X_{H^i}$ is $\tau$-related to $L_i$ 
one sees that   
\begin{equation}\label{e:proj-version}
 u' 
 = T\tau.\gamma'
 = T\tau.\mathcal{P}.\sum X_{H^i}(\gamma)\xi_i
 = \sum L_i(u)\xi_i
\end{equation}
whence at this projected level one cannot distinguish the Hamiltonian
from the non-holonomic version of the control problem     
\eqref{e:control-problem}. 
However, accepting the non-holonomic nature of the construction of
Brownian motion one can give an almost Hamiltonian description in
terms of the 
diffusion $\Gamma^{\mathcal{C}}$ in $\M$ defined by the 
Stratonovich equation 
\[
 \delta\Gamma_t^{\mathcal{C}} =
 \mathcal{S}^{\mathcal{C}}(\Gamma_t,W_t,\delta W_t), 
 \qquad
 \Gamma_0 = \eta_u\in\M \textup{ a.s.} 
\]
where the $\C$-valued Stratonovich operator is
given by 
\begin{equation}\label{e:S-C}
 \mathcal{S}^{\mathcal{C}}:
 \M\times T\R^n\longto \C\subset T\M,
 \quad
 (\eta_u,w,w')
 \longmapsto
 \sum \mathcal{P}(\eta_u)X_{H^i}(\eta_u)\vv<e_i,w'>.
\end{equation}
Of course, by \eqref{e:proj-version}, we retain the basic feature
that $\rho\circ\tau\circ\Gamma^{\mathcal{C}}$ is Brownian motion in
$Q$. But now we also obtain an induced diffusion in $T^*Q$.
Let $\pi^{T^*Q}: T^*Q\toto(T^*Q)/G$ 
and $\pi^{\mathcal{M}}: \M\toto \M/K=T^*Q$ be the 
orbit projections. 

\begin{proposition}
The diffusion $\Gamma^{\mathcal{C}}$ induces a diffusion
$\Gamma^{T^*Q} = \pi^{\mathcal{M}}\circ\Gamma^{\mathcal{C}}$ in 
$\M/K = J_K^{-1}(0)/K = T^*Q$
as well as a diffusion $\Gamma^{\textup{red}} =
\pi^{T^*Q}\circ\Gamma^{T^*Q}$ 
in $\M/(K\times G) =
(T^*Q)/G$. Moreover, $\tau\circ\Gamma^{T^*Q}$ is Brownian motion in
$(Q,\mu)$. 
\end{proposition}

\begin{proof}
Since $\mathcal{S}^{\mathcal{C}} = \mathcal{P}.\mathcal{S}^{T^*P}$ and 
the projection $\mathcal{P}$ is $K$-equivariant we preserve the
$K$-equivariance of $\mathcal{S}^{T^*P}$. It is also easy to see that
$\mathcal{P}$ is $G$-equivariant whence $G$-invariance of
$\mathcal{S}^{T^*P}$ is preserved as well. 
Hence
Theorem~\ref{thm:e-r} applies to the $K$-action as well as to the
$K\times G$-action on $\M$. 
\end{proof}

However, it is not clear that the induced diffusion
$\Gamma^{\textup{red}}$ will preserve the symplectic leaves
$T^*Q\sporb G\subset (T^*Q)/G$.\footnote{It is slightly imprecise to
  speak of symplectic leaves in this context: The $T^*Q\sporb G$ need
  not be connected, they need not even be smooth manifolds. It is
  rather the connected components of the smooth strata of $T^*Q\sporb
  G$ which should be called symplectic leaves.} 
The problem is that, in non-holonomic
mechanics, symmetries need not lead to conservation laws, and
correspondingly it cannot be asserted that
$\mathcal{S}^{\mathcal{C}}(\eta_u,w,w')\in T_{\eta_u}(J_{K\times
  G}^{-1}(0\times\orb))$ for $\eta_u\in J_{K\times G}(0\times\orb)$. 
Instead of a conservation law one obtains a momentum equation, see \cite{CMR01}.

\subsection{The flat case}
If the manifold $(Q,\mu)$ is flat so that it admits a global orthonormal
frame $L_a\in\X(Q)$, $a=1,\dots,n$ then one can give a Hamiltonian
construction of Brownian motion in $Q$ which is formulated on $T^*Q$
(\cite{LO08}): 
Define Hamiltonian  momentum functions on $T^*Q$ by
\[
 H^0(q,p) = -\by{1}{2}\vv<p,\sum \nabla^{\mu}_{L_a}L_a(q)>
 \textup{ and }
 H^a(q,p) = \vv<p,L_a(q)>.
\]
Consider the Stratonovich equation $\delta\Gamma^{T^*Q} =
\mathcal{S}^{T^*Q}(\Gamma^{T^*Q},W,\delta W)$ where $W$ is Brownian
motion in $\R^n$ and the Stratonovich operator is defined as
\begin{equation}\label{e:BM-Ham-flat}
 \mathcal{S}^{T^*Q}:
 T^*Q\times T(\R\times\R^n)  \longto T(T^*Q),\quad
 (q,p;t,t',w,w')
 \longmapsto
 X_{H^0}(q,p)t' + \sum X_{H^a}(q,p)\vv<e_a,w'>.
\end{equation}
Then $\Gamma^Q := \tau\circ\Gamma^{T^*Q}$ is a diffusion in $Q$ with generator
$\by{1}{2}\sum(L_a L_a - \nabla^{\mu}_{L_a}L_a) =
\by{1}{2}\Delta^{\mu}$, 
i.e., Brownian motion.  
If a Lie group $G$ acts properly and by isometries on $(Q,\mu)$ and the
frame $(L_a)$ is $G$-invariant then 
$\Gamma^{T^*Q}$ drops to a diffusion in the (singular) Poisson
quotient $(T^*Q)/G$ which preserves the 
symplectic leaves $T^*Q\sporb G = J_G^{-1}(\orb)/G$.

\section{Reduction of Brownian motion with respect to
  polar actions}\label{sec:BM-polar}

We give two different Hamiltonian constructions of Brownian
motion on a Riemannian $G$-manifold $(Q,\mu)$, and discuss their respective
symmetry reductions.
Throughout $T^*Q$ will be endowed with the canonical exact symplectic
form denoted by $\Om^Q$ or $\Om$ if no confusion is possible.

\subsection{Generalities on transformation groups}\label{sec:3A}
Consider a proper Riemannian $G$-manifold $(Q,\mu)$. This means that
$G$ is a Lie group acting properly and by isometries on the Riemannian
manifold $Q$. 
We will assume that the $G$-action is of \emph{single orbit type}
which means that there is a subgroup $H\subset G$ such that, for any
$q\in Q$, the stabilizer subgroup $G_q$ is conjugate to $H$ within
$G$. This has the consequence that $B:=Q/G$ is a smooth manifold and
the projection map $\pi: Q\toto B$ is a fiber bundle, albeit not a
principle one. Its typical fiber is of the form $G/H$. See \cite{PT88}.
 
The vertical bundle $\ver\subset TQ$ on $\pi: Q\toto B$ is defined as 
\[
 \ver := \ker T\pi.
\]

\begin{definition}
The (generalized) \emph{mechanical connection} on $\pi: Q\toto B$ is
the fiber bundle connection which is specified by requiring the
horizontal bundle $\hor$ to be orthogonal to the vertical one with respect to
the $G$-invariant Riemannian metric $\mu$.\footnote{We have included
  the prefix `generalized' because usually the mechanical connection
  is defined on a principal bundle. In order not to overload the
  nomenclature we will subsequently omit this prefix.}  
\end{definition}

See \eqref{e:char-of-A} below for an alternative characterization of
the mechanical connection.

From now on we equip $\pi: Q\toto B$ with the mechanical connection.
Let $\A: TQ\to\gu$ be the associated connection
form. This is not a principal bundle connection form. 
For $q\in Q$ with isotropy group $G_q$ we have 
\[
 A_q\zeta_X(q) = X
 \textup{ for all }
 X\in\gu_q^{\bot}.
\]
Here $\gu_q\subset\gu$ is the infinitesimal stabilizer at $q$, $\zeta:
\gu\to\X(Q)$ is the fundamental vector field mapping, $\gu$ is
equipped with an $\Ad$-invariant inner product $\vv<.,.>$ and
$\gu_q^{\bot}$ is the $\vv<.,.>$-orthogonal to $\gu_q$.
At a point $q\in Q$ we have $\zeta_X(q) =
\dd{t}{}|_0\exp(tX).q \in T_qQ$. 
We may view $A$ also as a $G$-equivariant bundle map
\begin{equation}\label{e:MC}
 A: TQ\longto \bsc_{q\in Q}\gu_q^{\bot}
\end{equation}
which restricts to an isomorphism on the vertical bundle.  
See \cite{Mon} for a definition in the context of a free action and
\cite{H04} for the generalization to single orbit type actions. 
Let $\xi_i = v_i+\zeta_{X_i} \in TQ=\hor\oplus\ver$ be decomposed into
horizontal and vertical parts. Then 
\[
 \mu(\xi_1,\xi_2)
 =
 \mu(v_1,v_2) + \I_q(X_1,X_2)
\]
where the $G$-invariant operator $\I: 
 \bsc_{q\in Q}\gu_q^{\bot}\times\bsc_{q\in Q}\gu_q^{\bot} \to\R$ is
defined by
\[
 \I_q(X_1,X_2) 
 := \mu_q(\zeta_{X_1}(q),\zeta_{X_2}(q)) 
 = \vv<A_q^*A_q X_1,X_2>,
\]
and is called the \caps{inertia tensor}.
For each $q\in Q$ we obtain a $G_q$-equivariant isomorphism $\I_q:
\gu_q^{\bot}\to\ann\gu_q$ where $\ann\gu_q$ is the annihilator of
$\gu_q$ in $\gu^*$. 

\begin{definition}[The canonical momentum map on $T^*Q$]
The map $J: T^*Q\to\gu^*$ defined by
\[
 \vv<J(q,p),X> = \vv<p,\zeta_X(q)>
\]
is the \emph{cotangent bundle momentum map}. 
\end{definition}

The mechanical connection $A$ can now be characterized by the
pointwise diagram
\begin{equation}\label{e:char-of-A}
\xymatrix{
 {T_qQ}
  \ar @{->}[r]^-{A_q}
  \ar @{->}[d]^{\mu_q}_{\cong}
 &
 {\gu_q^{\bot}}
  \ar @{->}[d]_-{\I_q}^{\cong}\\
 {T^*_qQ}\ar @{->}[r]^-{J_q}
 &
 {\ann\gu_q}
}
\end{equation} 
where $J_q$ is the restriction of $J$ to $T_q^*Q$, whence its
importance for mechanical systems.

Let us now assume additionally that $G$ is compact and acts by polar
transformations on $Q$. 
This means that 
there is a submanifold $M\subset Q$ which meets all group orbits in an
orthogonal manner. 
The submanifold $M$ is called a (global cross-) section of the action. 
See \cite{PT88}. 
A canonical example is the conjugation action of a compact Lie group
on itself. In this case a section is given by a maximal torus. 
As with
this example, it is not required that the action be free.
Note that there is a residual action by 
$W:=\set{g\in G: g.M\subset M}/\set{g\in G: g.x=x\;\forall x\in M}$ on
$M$ and that
\[
 Q/G = M/W =: B. 
\]
In the aforementioned example $W$ coincides with the
Weyl group of the compact Lie group and $B$ is the interior of a Weyl chamber.

We cocnctinue to  assume that the $G$-action is of single orbit type. In the
canonical example where the group acts upon itself by conjugation this
amounts to passing to the open dense subset of regular points. 
 
There is thus a local diffeomorphism $Q\cong M\times G/H$. This
diffeomorphism is generally not global: Consider $\SO(3)$ acting on
$Q=\R^3\setminus0$ then $M=\R\setminus0$ and $G/H=S^2$. Factoring out
the residual $W$-action one does obtain a global diffeomorphism
\[
 Q\cong M\times_W G/H.
\]
By reason of dimension it follows that $W$ is discrete and hence
finite since $G$ is compact. Since $Q/G
= M/W = B$ it is also true that $M \cong \bsc_{w\in W}B$. Thus there
is also a (non-canonical) diffeomorphism 
\begin{equation}\label{e:non-can}
 Q\cong B\times G/H.
\end{equation}
Moreover, we will assume that the action is actually hyper-polar: $M$ is
supposed to be locally isometrically diffeomorphic to a Euclidean
space $\R^l$ such that $B\cong\R^l$.

\subsection{Deterministic Hamiltonian reduction}
Let $\Ham: T^*Q\to\R$ be the $G$-invariant kinetic energy Hamiltonian
associated to $\mu$. By \cite{H04}
Hamiltonian reduction of
$(T^*Q,\Om^Q,\Ham)$ at the orbit level  $\orb$  yields the reduced stratified
Hamiltonian system  
$(T^*Q\sporb G,\Om^{\mathcal{O}},\Ham^{\mathcal{O}})$;
using the mechanical connection $A$ the phase space of 
this system can be realized as 
\begin{equation}
 T^*Q\sporb G\cong T^*B\times_B(\bsc_{q\in Q}\orb\cap\ann\gu_q)/G,
 \quad [(q,p)]\longmapsto (\pi(q),\eta^*_q(p);[(q,A_q^*(\lam))])
\end{equation}
where $\eta_q: T_xB \to T_qQ$ is the horizontal lift mapping associated
to $A$.
The reduced Hamiltonian then becomes
\begin{equation}\label{e:H^0}
 \Ham^{\mathcal{O}}:  
 (x=\pi(q),u;[(q,\lam)])
 \longmapsto
 \by{1}{2}\vv<u,u> + \by{1}{2}\vv<\lam,\I_q^{-1}(\lam)>.
\end{equation}  
The reduced space $T^*B\times_B(\bsc_{q\in Q}\orb\cap\ann\gu_q)/G$ is
a stratified symplectic fiber bundle over $T^*B$ with standard fiber
$(\orb\cap\ann\ho)/H = \orb\spr{0}H$ 
which is a stratified symplectic space. 
This is
the (singular) bundle picture in mechanics. 
The potential term $\by{1}{2}\vv<\lam,\I_q^{-1}(\lam)>$ can also be
viewed as a $W$-invariant function on $M\times\orb\spr{0}H$.

\subsection{Brownian motion in a constant frame and reduction}
The situation we have in mind in this subsection is that of a mechanical
system defined on a Lie group $G$ such that the (kinetic energy) Hamiltonian is
invariant under the conjugation action and the tangent space is
trivialized as $TG = G\times\gu$ via (left or right)
multiplication. 
Prototypical examples for this set-up are the Calogero-Moser systems
discussed 
in Section~\ref{sec:CMS}. 
This set-up is quite general in the sense that the hierarchy of
Calogero-Moser systems is very rich;
any (real or complex)
semi-simple or also reductive Lie group can be taken as a
configuration space, and different choices will lead to different
(versions of Calogero-Moser) dynamical systems. The classical
Calogero-Moser system corresponds to the choice $G=\SU(n)$ of
\cite{KKS78}. Generalizations where $G$ is a loop group could also be
feasible and will give rise to a Calogero-Moser system with an
elliptic interaction potential, see e.g.\ \cite{GN94}.

Suppose $Q$ carries a global orthonormal frame $L_A$, $A=1,\dots,n$
such that
$\nabla^{\mu}_{L_A}L_A = 0$ and
 $TQ\cong Q\times\qo$ via this frame with $\qo = \R^n$. 
A Hamiltonian construction $\Gamma^{T^*Q}$ of Brownian motion
$\Gamma^Q=\tau\circ\Gamma^{T^*Q}$  can then be given in
terms of the Stratonovich equation associated to operator
$\mathcal{S}^{T^*Q}$ defined in 
\eqref{e:BM-Ham-flat}. 
Since $H_{w'}(q,p) = \sum H^{A}(q,p)\vv<e_A,w'> = \vv<p,w'> =
H_{gw'}(gq,gp)$ for $g\in G$ it follows as in the proof of 
Lemma~\ref{lm:equiv-P} that $\mathcal{S}^{T^*Q}$ is $G$-equivariant. 
Therefore, by Theorem~\ref{thm:e-r} the diffusion $\Gamma^{T^*Q}$
factors through the projection $\pi^{T^*Q}: T^*Q\toto (T^*Q)/G$ to 
$\pi^{T^*Q}\circ\Gamma^{T^*Q}$ in the quotient. 
Consider the diagram 
\begin{equation}\label{e:diag}
\xymatrix{
 {T^*Q}
  \ar @{->>}[r]^-{\pi^{T^*Q}}
  \ar @{->}[d]_{\tau}
 &
 {(T^*Q)/G}
  \ar @{->}[d]^-{\tau/G}\\
 {Q}\ar @{->>}[r]^-{\pi}
 &
 {Q/G}
}
\end{equation}
We can alternatively write Brownian motion on $(Q,\mu)$ as the
solution to the Stratonovich diffusion equation $\delta\Gamma^Q = \sum
L_A\delta W^A$. Again this operator is $G$-equivariant and, repeating
the same reasoning, we obtain an induced diffusion
\[
 \Gamma^B = \pi\circ\Gamma^Q =
 (\tau/G)\circ\pi^{T^*Q}\circ\Gamma^{T^*Q}
\]
in $B=Q/G$. 
In order to obtain the generator $\by{1}{2}\Delta^0$ of $\Gamma^B$
we
note that the generator of $\Gamma^Q$ is $\by{1}{2}\Delta^Q$ and
introduce the following function:  
Let $e_{\alpha}$ be a an orthonormal basis on $\gu$ and define $\delta:
Q\to\R_{>0}$ by
\begin{equation}\label{e:delta}
 \delta(q)
 = 
 \Big|\det(\I_q(e_{\alpha},e_{\beta}))_{\alpha\beta}\Big|^{\frac{1}{2}}.
\end{equation}
This function is $G$-invariant and we will call the induced function
on $B$ also $\delta$. 
The projection $\pi: Q\toto B$ is a Riemannian submersion with
respect to the induced metric $\mu_0$ on $B$. Let $\Delta^Q$, $\Delta^B$ denote
the Laplace-Beltrami operator on $(Q,\mu)$, $(B,\mu_0)$ respectively.
Note that $\Delta^Q$ acts on the set of $G$-invariant functions 
$\cinf(Q)^G$.

\begin{proposition}
Let $\by{1}{2}\Delta^0$ be the induced operator on $\cinf(B)$ characterized by
$\pi^*\circ\Delta^0 = \Delta^Q\circ\pi^*$. Then
\begin{equation}\label{e:Delta_0}
 \Delta^0 = \Delta^B + \nabla\log\delta
\end{equation}
where $\nabla$ is the $\mu_0$-gradient. 
\end{proposition}

\begin{proof}
This follows either from the coordinate expression $\Delta^Q = \sum
g^{-\frac{1}{2}}\del_i\sqrt{g}g^{ij}\del_j$ in suitably choosen
local coordinates, or by appealing to \cite{H72}.
\end{proof}

We view $\Gamma^B$ as the reduction of the probabilistic version of a
free particle motion on $(Q,\mu)$.

It is known that, if $\Gamma^B$ is
\emph{critical in the sense of Guerra and Morato} for
the stochastic action functional associated to the Lagrangian 
$
 L: TB\to \R
$,
\[
 L(q,v) = \by{1}{2}\mu_0(v,v) - V(q),
\]
then, with $\hbar = 1$, 
the stochastic Hamilton-Jacobi equation of 
Nelson~\cite[p.~72, Equ.~(14.17)]{N85}
\begin{equation}\label{e:SHJE}
 V - \by{1}{2}\mu_0(b,b) - \by{1}{2}\textup{div}\, b = 0
\end{equation}
holds. 
Note that by \eqref{e:Delta_0} and Definition~\ref{def:drift},
$b = \by{1}{2}\nabla(\log\delta)$ is the drift of the 
diffusion $\Gamma^B$ (with respect to the Levi-Civita connection
associated to $\mu_0$.). 
See also Guerra and Morato~\cite{GM83},
or  Zambrini~\cite[Equ.~(3.14)]{Z09} for a version in terms of
Yasue's least Action Principle.
The drift $b$ is a pure osmotic velocity in the nomenclature of
\cite{N85}. 

\begin{remark}
We have arrived at \eqref{e:SHJE} by asking the the following
question: Is there a Lagrangian $L=\by{1}{2}||\cdot||_{\mu_0}-V$ on
$TB$ such that the induced diffusion $\Gamma^B$ is critical for $L$ in
the sense of Guerra-Morato? By \cite[Thm.~14.1]{N85} this is
equivalent to solving $v=\nabla S$ where $v$ is the current velocity
and 
\[
 S(x,t) =
 -E\big[\int_t^{t_1}L_+(\Gamma^B_s,s)\,ds\big|\Gamma^B_t=x\big]
\]
and $L_+(x,s) = \by{1}{2}\mu_0(b,b)_x +
\by{1}{2}\textup{div}(b)_x-V(x,s)$. Given the diffusion (and its drift) this
amounts to finding a suitable potential function $V$, and the simplest
possible choice is expressed by \eqref{e:SHJE}. That is, $L_+=0$ and
$v=0$. Since the drift is the sum of osmotic and current velocity it
remains to verify that $b=\by{1}{2}\nabla(\log\delta)$ satisfies the
osmotic equation
\[
 \by{1}{2}\Delta\rho = \textup{div}(\rho b)
\]
and this certainly holds with $\rho=\delta$ and
$b=\by{1}{2}\nabla(\log\delta)$. 
We emphasize that this corresponds to the simplest possible choice for
a potential. There could be more interesting choices, arising
(probably) from a spin-debendent potential.     
\end{remark}

Thus we use \eqref{e:SHJE} as the defining equation for $V$ and find
\begin{equation}\label{e:V}
 V
 = \by{1}{4}(\Delta^B\log\delta 
   + \by{1}{2}\mu_0(\nabla\log\delta,\nabla\log\delta)
 = \by{1}{2}\delta^{-\frac{1}{2}}\Delta^B\delta^{\frac{1}{2}}
\end{equation}

With 
\[
 \psi = \delta^{\frac{1}{2}}
\]
and $b = \by{1}{2}\nabla(\log\delta)$
equation~\eqref{e:SHJE}
is equivalent to the stationary Schr\"odinger equation $H\psi = 0$
with the Hamiltonian operator
\begin{equation}\label{e:H}
 H = -\by{1}{2}\Delta^B + V;
\end{equation}
indeed
\begin{align*}
 \delta^{-\frac{1}{2}}H\delta^{\frac{1}{2}}
 &=
 \delta^{-\frac{1}{2}}\Big(-\by{1}{2}\Delta\delta^{\frac{1}{2}}+V\delta^{\frac{1}{2}}\Big)
 =
 \delta^{-\frac{1}{2}}\Big(-\by{1}{2}\textup{div}(\by{1}{2}\delta^{-\frac{1}{2}}\nabla\delta)+V\delta^{\frac{1}{2}}\Big)\\
 &=
 \by{1}{8}\delta^{-2}\mu_0(\nabla\delta,\nabla\delta)-\delta^{-1}\Delta\delta+V
 =
 \by{1}{2}\mu_0(b,b)-\by{1}{4}\delta^{-2}\mu_0(\nabla\delta,\nabla\delta)-\by{1}{4}\textup{div}(\nabla\log\delta)+V\\
 &=
 -\by{1}{2}\mu_0(b,b)-\by{1}{2}\textup{div}(b)+V.
\end{align*}

To compare these formulas with the case of quantum reduction of a free
particle under polar actions we quote a special case of a theorem of
\cite[Thm.~4.5]{FP08}.

\begin{theorem}[Feher and Pusztai~\cite{FP08}, spinfree version]\label{thm:FP}
The reduction of the (spinless) quantum system defined by the closure
of $-\by{1}{2}\Delta^Q$ on $C^{\infty}_{\textup{cp}}(Q)\subset
L^2(Q,d\mu)$ leads to the reduced Hamiltonian operator
\begin{equation}\label{e:FP}
 H_{\textup{QM}} = H
\end{equation}
with $H$ as in \eqref{e:H}.\footnote{The Planck constant $\hbar$ could
  be introduced by rescaling the diffusion operator on $Q$, i.e., the
  metric $\mu$. This rescaling does not affect the potential $V$ since
  factor cancels in this term.}
This operator is essentially self-adjoint on a suitable domain
(specified in \cite{FP08}).
\end{theorem}

We note that this theorem, as it is stated, is actually only the zeroth
order version of 
\cite[Thm.~4.5]{FP08} since
in its full version  it also 
incorporates a
spin dependent potential energy term. 
In the next section we give a different approach where such a term 
can also be included 
in the stochastic context.

\subsection{Brownian motion in a $G$-adapted frame and reduction}  
Assume the quotient $B=M/W=Q/G\cong\R^l$ is Euclidean and denote its $x$-independent orthonormal
basis by 
\begin{equation}\label{e:v}
 v_1,\dots,v_l. 
\end{equation}
Let 
\begin{equation}\label{e:Y}
 Y_1(x),\dots,Y_k(x)
\end{equation} 
denote an
$\I_x$-orthonormal frame on $\ho^{\bot}\subset\gu$
which depends smoothly on $x\in M$.
(Since $M=\bsc_{w\in W}B$ as argued in Section~\ref{sec:3A}  
we may non-canonically embed $B$ in $M$ as
an open subset and
thus we use the same variable  $x$ for elements in $B$ as well as in
$M$.) 
Concerning indices
we make the
convention that 
\[
 1\le A,B,C \le n=\dim Q,\quad
 1\le i,j,k \le l=\dim M,\quad
 1\le \alpha,\beta,\gamma \le k =\dim(G/H).
\] 
Then using \eqref{e:non-can}
$(L_A)_A$ is an orthonormal frame on
$Q\cong B\times G/H$ 
where we define 
\[
 L_A(q) 
 = L_A(g.x)
 = 
 \left\{
 \begin{matrix}
  g.v_i && \textup{if } A=i;\\
  g.\zeta_{Y_{\alpha}(x)}(x) && \textup{if }A=\alpha+l.
 \end{matrix} 
 \right\}
\] 

Now Brownian motion in $(Q,\mu)$  can be constructed via 
$\Gamma^Q = \tau\circ\Gamma^{T^*Q}$ as in 
\eqref{e:BM-Ham-flat}. 
Thus we need to calculate $H^0$: 
take local coordinates $(q^A)$ around a point $q=g.x$ which are adapted to the decomposition
$T_qQ\cong T_xM\oplus T_{[g]}G/H$ and express $L_A = \sum L_A^B\dd{q^b}{}$ whence
\begin{align*} 
 \sum\nabla^{\mu}_{L_A}L_A 
 &= 
 \sum L_A^B\nabla^{\mu}_{\frac{\del}{\del q^B}}(L_A^C\frac{\del}{\del q^C} )
 =
 0 + \sum L_A^BL_A^C\Gamma^D_{BC}\frac{\del}{\del q^D}
 =
 \sum \mu^{BC}\Gamma^D_{BC}\frac{\del}{\del q^D}\\
 &=
 - \frac{1}{\sqrt{|\mu|}}\frac{\del(\sqrt{|\mu|}\mu^{DE})}{\del q^D} \frac{\del}{\del q^D}
 = - \nabla\log\delta
\end{align*} 
where $\delta = \sqrt{|\mu|}$ was defined in \eqref{e:delta}.  
Therefore, $H^0(q,p) = \by{1}{2}\vv<p,\nabla\log\delta>$.

The Hamiltonian diffusion $\Gamma^{T^*Q}$ in $T^*Q$ is
defined by the Stratonovich equation
\begin{equation}
 \delta\Gamma^{T^*Q}
 =
 X_{H^0}(\Gamma^{\mathcal{H}})\delta t
 + \sum X_{H^A}(\Gamma^{\mathcal{H}})\delta W^A
\end{equation}
with notation as in \eqref{e:BM-Ham-flat}. 
The associated Stratonovich operator 
$T^*Q\times T(\R\times\R^n)\to T(T^*Q)$, 
$(q,p,t,t',w,w')\mapsto X_{H^0}(q,p)t' + \sum X_{H^A}(q,p)\vv<e^A,w'>$ 
is clearly
$G$-\emph{invariant} and reduces to an operator
$\mathcal{S}^{\mathcal{O}}$ on the reduced phase
space $T^*Q\sporb G$. In particular, if $\pi^{T^*Q}: T^*Q\toto(T^*Q)/G$
is the orbit space projection we have:

\begin{proposition}
$\Gamma^{T^*Q}$ drops to a
diffusion $\pi^{T^*Q}\circ\Gamma^{T^*Q}$ 
in $(T^*Q)/G$ which, for every coadjoint orbit
$\orb\subset\gu^*$,  preserves the smooth connected
symplectic strata in $T^*Q\sporb G\subset (T^*Q)/G$.
\end{proposition}

\begin{proof}
On each smooth connected symplectic stratum $L\subset T^*Q\sporb G$ we can express
$\pi^{T^*Q}\circ\Gamma^{T^*Q}$ as the solution to the
Hamiltonian Stratonovich equation
$\delta(\pi^{T^*Q}\circ\Gamma^{T^*Q}) 
=
\mathcal{S}^L(\pi^{T^*Q}\circ\Gamma^{T^*Q},t,\delta
t,W, \delta W)$.
The reduced operator $\mathcal{S}^L$ on $L$ is given by 
\[
 \mathcal{S}^L:
 L\times T(\R\times\R^n)
 \longto TL,
 \quad
 (\eta,t,t',w,w')
 \longmapsto 
 X_{h^0}(\eta)t'
  + \sum X_{h^A}(\eta)\vv<e_A,w'>
\]
where $h^0$, $h^A$ are the induced functions on $L$
which pull-back to the restrictions of $H^0$, $H^A$ to
$(\pi^{T^*Q})^{-1}(L)$. See also \cite{LO08a}.
\end{proof}

Let us denote the restriction of $\pi^{T^*Q}\circ\Gamma^{T^*Q}$
to $T^*Q\sporb G$ by $\Gamma^{\mathcal{O}}$. According to the above
this diffusion further restricts to each smooth symplectic stratum
$L\subset T^*Q\sporb G$. 

To obtain formulas for $h^0$ and  $h^A$ we 
use the $A$-dependent $G$-equivariant isomorphism
\begin{align*}
 \Phi: {}\;
 \W := Q\times_BT^*B\oplus\bsc_{q\in Q}\ann\gu_q\longto T^*Q,\quad
 \big( (q;x,u),(q,\lam) \big)
 \longmapsto
 \big( q,(T_q^*\pi)(u) + A_q^*(\lam) \big).
\end{align*}
(This isomorphism depends on the choice of the connection and works
only if the action is of single orbit type. However, it does not
depend on whether or not the action is polar.)
Thus we find, denoting the function on $B$ induced from
\eqref{e:delta} again by $\delta$,
\begin{align*}
 \Phi^*H^0\big( (q;x,u),(q,\lam) \big)
 &= \vv<(T_q\pi)^*(u) + A_q^*(\lam), -\by{1}{2}\sum\nabla^{\mu}_{L_A}L_A>
 = \by{1}{2}\vv<u,\nabla\log\delta(x)>,\\
 \Phi^*H^i\big( (q;x,u),(q,\lam) \big)
 &=
 \vv<u,v_i>,\\
 \Phi^*H^{\alpha}\big( (q;x,u),(q,\lam) \big)
 &= \vv<A_q^*(\lam), g.\zeta_{Y_{\alpha}(x)}(x)>
 = \vv<g^{-1}.\lam, Y_{\alpha}(x)>
\end{align*}
where we used the equivariance property of $A$, and $x$, $q$ and $g$ are
related by $q=g.x$. The
basis elements $v_i$ and $Y_{\alpha}$ are introduced 
in \eqref{e:v} and \eqref{e:Y}.

Because we want to apply the stochastic Hamilton-Jacobi equation of
Lazaro-Cami and Ortega \cite{LO09} we want the reduced space $T^*Q\sporb G$
to be a cotangent bundle. Therefore, we assume now that
\[ 
 \orb\spr{0}H = \set{\textup{point}} = \set{[\lam]_H}
\]
whence $T^*Q\sporb G = T^*B\times\set{[\lam]_H} = T^*B$.
This is a very restrictive condidtion and cases where this assumption
holds are very important in the theory of Calogero-Moser systems. See
\cite{FP06a}
for a classification, further information and the so-called `KKS-mechanism'.
The reduced Hamiltonian functions thus become
\begin{align*}
 h^0(x,u) 
 &= \by{1}{2}\vv<u,\nabla\log\delta(x)>,\\
 h^i(x,u) 
 &=
 \vv<u,v_i>,\\
 h^{\alpha}(x,u) 
 &= \vv<\lam, Y_{\alpha}(x)>.
\end{align*}
Let $f\in\cinf(T^*B)$. The stochastic Hamilton-Jacobi equation of
\cite{LO09} associated to the reduced Hamiltonian diffusion
$\Gamma^{\mathcal{O}}$, given by  
\begin{equation}\label{e:Ham-O}
\delta\Gamma^{\mathcal{O}} = X_{h^0}(\Gamma^{\mathcal{O}})\delta t 
+ \sum X_{h^i}(\Gamma^{\mathcal{O}})\delta W^i 
+ \sum X_{h^{\alpha}}(\Gamma^{\mathcal{O}})\delta W^{\alpha},
\end{equation}
and the Lagrangian submanifold $L_f=\textup{graph}\,df\subset T^*B$ is
an equation of semi-martingales which reads
\begin{equation}\label{e:SHJE-LO}
 \widetilde{S}^x_t(\om)
 = f(x) 
  - \int_0^t h^0(x,d\widetilde{S}^x_s(\om))\;ds 
  - \int_0^t h^i(x,d\widetilde{S}^x_s(\om))\;\delta W^i_s
  - \int_0^t h^{\alpha}(x,d\widetilde{S}^x_s(\om))\;\delta W^{\alpha}_s,
 \qquad\textup{a.s.}
\end{equation}
Here $\widetilde{S}^x$ is, for all $x\in B$, a continuous $\R$-valued
semimartingale defined on some underlying probability space. The map
$B\to\R$, $x\mapsto \widetilde{S}^x_t(\om)$ is $C^1$ where it is well-defined
and accordingly $d\widetilde{S}^x_t(\om) = d(x\mapsto\widetilde{S}_t^x(\om))$. These
statements are proved in \cite{LO09} where $\widetilde{S}$ is called the
\emph{projected stochastic action} and is constructed in terms of the
defining data of the Hamiltonian diffusion and the Lagrange
submanifold $L_f$. 
In the following lines we parallel the arguments of
\cite[Example~1]{LO09} adapted to \eqref{e:SHJE-LO}.\footnote{There
  are two small differences: We use the convention that $[W^i,W^j] =
  \delta^{ij}t$, and there is a constant factor in \cite[Example~1]{LO09} which we could not
  verify whence some other choices differ as well.}
Thus we use the conversion rule \cite[p.~81]{Pro} that transforms Ito
to Stratonovich equations and find, a.s.,
\begin{align*}
 \widetilde{S}^x_t(\om)
 &= f(x) 
  - \int_0^t \by{1}{2}\vv<\nabla\widetilde{S}^x_s(\om),\nabla\log\delta> \;ds 
  - \sum\int_0^t \frac{\del\widetilde{S}^x_s(\om))}{\del x^i}\;\delta W^i_s
  - \sum\vv<\lam,Y_{\alpha}(x)>\int_0^t \;\delta W^{\alpha}_s\\
 &=
  f(x) 
  + \int_0^t\big( 
             \by{1}{2}\Delta\widetilde{S}^x_s(\om)
             - \by{1}{2}\vv<\nabla\widetilde{S}^x_s(\om),\nabla\log\delta>
             \big)\;ds 
  - \sum\int_0^t \frac{\del\widetilde{S}^x_s(\om))}{\del x^i}\;d W^i_s
  - \sum\vv<\lam,Y_{\alpha}(x)> W^{\alpha}_t.\\
\end{align*}
For the quadratic variation $[\widetilde{S}^x_t,\widetilde{S}^x_t]$ (see
\cite[Thm.~II.29]{Pro}) this implies
\[
 \by{1}{2}[\widetilde{S}^x_t,\widetilde{S}^x_t]
 = 
 \by{1}{2}\int_0^t\big(
  \vv<\nabla\widetilde{S}^x_s,\nabla\widetilde{S}^x_s> + \sum \vv<\lam,Y_{\alpha}(x)>^2
 \big) \;ds
 =
 \int_0^t\mathcal{H}^{\mathcal{O}}(x,\nabla\widetilde{S}^x_s) \;ds
\]
since 
$\sum \vv<\lam,Y_{\alpha}(x)>^2 =
\vv<\lam,\I_x^{-1}(\lam)>$ which equals twice the potential term
\eqref{e:H^0} in
the reduced deterministic Hamiltonian $\mathcal{H}^{\mathcal{O}}$. 
Applying the Ito formula \cite[Thm.~II.32]{Pro} to the semi-martingale
$\exp(-\widetilde{S}^x)$ thus yields
\begin{align*}
 e^{-\widetilde{S}^x_t} - e^{-f(x)}
 &=
   -\int_0^t e^{-\widetilde{S}^x_s}\;d\widetilde{S}^x_s
  + \by{1}{2}\int_0^t e^{-\widetilde{S}^x_s}\;d[\widetilde{S}^x_s,\widetilde{S}^x_s]\\
 &=
  \int_0^t e^{-\widetilde{S}^x_s}
    \big(
             - \by{1}{2}\Delta\widetilde{S}^x_s(\om)
             + \by{1}{2}\vv<\nabla\widetilde{S}^x_s(\om),\nabla\log\delta>
             + \by{1}{2}\vv<\nabla\widetilde{S}^x_s(\om),\nabla\widetilde{S}^x_s(\om)>
             + \by{1}{2}\vv<\lam,\I_x^{-1}(\lam)>
    \big)\;ds   \\
 &\phantom{===}
  + \sum\int_0^t e^{\widetilde{S}^x_s}\frac{\del\widetilde{S}^x_s(\om))}{\del x^i}\;d W^i_s
  + \sum\vv<\lam,Y_{\alpha}(x)> \int_0^t e^{\widetilde{S}^x_s}\; dW^{\alpha}_s.\\
\end{align*}
Observe that
$
 e^{-\widetilde{S}^x_s}(
   -\Delta\widetilde{S}^x_s(\om)  
   +
   \vv<\nabla\widetilde{S}^x_s(\om),\nabla\widetilde{S}^x_s(\om)>
   +
   \vv<\nabla\widetilde{S}^x_s(\om),\nabla\log\delta>)
 =
 \Delta e^{-\widetilde{S}^x_s} - \vv<\nabla e^{-\widetilde{S}^x_s(\om)},\nabla\log\delta>
$.
With 
\[
 \psi(t,x) := \delta^{-\frac{1}{2}}(x)E[\exp(-\widetilde{S}^x_t)],
\]
and assuming sufficient regularity conditions to interchange the order
of integration,
we thus
find the following diffusion equation: 

\begin{proposition}
\begin{equation}\label{e:H1}
 \frac{\del}{\del t}{}\psi
 = 
 \Big(\by{1}{2}\Delta 
 - \by{1}{2}\delta^{\frac{1}{2}}\Delta\delta^{-\frac{1}{2}} 
 + \by{1}{2}\vv<\lam,\I_x^{-1}(\lam)>\Big) \psi.
\end{equation}
\end{proposition}

As with \eqref{e:H} and \eqref{e:V}  this should be compared with
\cite[Thm.~4.5]{FP08}.
One could repeat the comments made after \cite[Thm.~4.5]{FP08}: The
first term on the right hand side of \eqref{e:H1} corresponds to the
classical kinetic energy while the third corresponds to the classical
potential in \eqref{e:H^0}, and the second term (``an extra measure
factor'') has no trace in the classical picture.
However, there is a sign difference in the exponent of this measure
factor. We will come back to this in the next section and attribute
it to be due to a time reversal in the construction of the projected 
stochastic action.  
We want to emphasize that the potential term
$\by{1}{2}\vv<\lam,\I_x^{-1}(\lam)>$ has purely stochastic origins
namely it appears by invoking the Ito formula, and yet it coincides
exactly with the classical potential that is obtained via Hamiltonian
reduction. 

Nelson~\cite{N85} uses the stochastic
  Hamilton-Jacobi equation of Guerra-Morato~\cite{GM83}
coupled with the Fokker-Planck equation to
  obtain an equation of the
  form $i\dd{t}{}\psi = H\psi$. While the ingredients which go into
  \eqref{e:H1} are `very similar' we could not produce the
  factor $i$ in this setting.

\subsection{Eigenfunctions of $\Delta^0$ and the projected stochastic action}\label{sub:zero-orb}
There were several loose ends in the previous sections: We have
introduced the approaches of Nelson \cite{N85} and Lazaro-Cami and
Ortega~\cite{LO09} to the stochastic Hamilton-Jacobi equation but not
said anything about their relation. 
Concerning equation \eqref{e:H1} one may wonder how properties of the
Lagrange submanifolds $L_f$ will relate to properties of solutions
$\psi$, and if stationary solutions 
$(\by{1}{2}\Delta 
 - \by{1}{2}\delta^{-\frac{1}{2}}\Delta\delta^{\frac{1}{2}} 
 + \by{1}{2}\vv<\lam,\I_x^{-1}(\lam)>)\psi=\gamma\psi$, for
$\gamma\in\R$, can be characterized in terms of $L_f$. We can
neither 
 answer these questions in general, 
nor give a thourogh examination of the relations between the different
approaches to the stochastic Hamilton-Jacobi equation.
To obtain some partial answers we
assume that $\orb\subset\gu^*$
 is the $0$-orbit,
that is
\begin{equation}\label{e:ass-lam}
 \lambda = 0,
\end{equation}
from now on.

Since we can identify $B$ and $\R^l$ we may  determine the Hamiltonian diffusion
$\Gamma^{\mathcal{O}} = (X_t,U_t)\in T^*B=\R^{2l}$  given by
\eqref{e:Ham-O} by the Ito equation
\[
 d\left(\begin{matrix}
  X_t^x\\
  U_t^u
 \end{matrix}\right)
 =
 \left(\begin{matrix}
  \frac{1}{2}\nabla\log\delta(X_t^x)dt + dB\\
  0
 \end{matrix}\right),
 \qquad
 \Gamma^{\mathcal{O}}_0 = (X_0^x,U_0^u) = (x,u)
 \textup{ a.s.}
\]
where $B$ is Brownian motion in $\R^l$.

Let $\xi(x)$ be the explosion time of a solution $X^x$ starting at $x$
a.s.
In \cite[Chapter~V]{Pro} it is shown that the following holds: 
$U_t(\om) := \set{x\in B: \xi(x,\om)>t} \subset B$ is
an open subset and  
$\textup{fl}_t(\om): U_t(\om)\to B$, 
$x\mapsto X^x_t(\om)$  
is a diffeomorphism onto its image. Hence, by the inverse function theorem,
for given $x$, $\om$ and (small) $t$ there
exists a unique point $\hat{x}(x,t,\om)$ such that
$X^{\hat{x}(x,t,\om)}_t(\om) = x$. 
Let us assume that for each $x$ there is an $\om$-independent number 
\[
 T_x>0 \textup{ with } T_x < \xi(x) \textup{ a.s.}
\]
such that $\hat{x}(x,t,\om)$ exists for all $t\le T_x$. (If not one
can use a stopping time but then the time-reversed process below needs
also
to be defined in terms of a stopping time, and this does not seem
to be standard.)

With assumption \eqref{e:ass-lam} the (projected) stochastic action 
of
\cite[Def.~3.2]{LO09} 
associated to the smooth function $f$ and
Lagrange submanifold $L_f = \set{(x,\nabla f(x)):
  x\in B} \subset TB = T^*B$ 
becomes
\begin{equation}\label{e:pr-ac}
 \widetilde{S}_t^x(\om)
 =
 f(\hat{x}(x,t,\om)) 
 + \int_0^t\big(i_{X_{h^0}}\theta - h^{0}\big)((\Gamma^{\mathcal{O}})_s^{z}(\om))\;ds
 + \sum  \int_0^t\big(i_{X_{h^i}}\theta - h^{i}\big)((\Gamma^{\mathcal{O}})_s^{z}(\om))\;dB^i_s
\end{equation}
where $z = z(x,t,\om) = (\hat{x}(x,t,\om),\nabla
f(\hat{x}(x,t,\om)))$  and $\theta$ is the Liouville one-form.  
Notice that $h^0$ and $h^i$ are momentum functions whence 
$i_{X_{h^0}}\theta - h^{0} = i_{X_{h^i}}\theta - h^i = 0$. 

We shall denote the time reversal of $X^x$ on the interval $[0,T_x]$
by $\hat{X}^x$. This is defined as 
\[
 \hat{X}^x_t := X^{\hat{x}(x,T_x,.)}_{T_x - t}, \quad t\in[0,T_x].
\]
By \cite{HP86,Pro} $\hat{X}^x$ is again a diffusion
and satisfies the Ito equation
\[
 d\hat{X}_t = -\by{1}{2}\nabla\log\delta(\hat{X}^x_t)dt + dB_t.
\]
Note that $\hat{X}^x_0(\om)
= X^{\hat{x}(x,T_x,\om)}_{T_x}(\om) = x$,
$\hat{X}^x_{T_x}(\om) = \hat{x}(x,T_x,\om)$
and
$\hat{X}^x_{t}(\om) = \hat{x}(x,t,\om)$
for $t\in[0,T_x]$. 
Thus we rewrite \eqref{e:pr-ac} and use the Ito formula to obtain 
\begin{align}\label{e:pr-ac1}
 \widetilde{S}^x_t
 &= f(\hat{X}^x_t)
  = f(x) + 
    \int_0^t \by{1}{2}\big(\Delta f 
        - \vv<\nabla f,\nabla\log\delta>
        \big)(\hat{X}^x_s)\;ds
    + \int_0^t \frac{\del f}{\del x^i}(\hat{X}^x_s)\;dB^i_s. 
\end{align}
Thus $E[\widetilde{S}^x_t]$ satisfies the Kolmogorov backward equation
$\dd{t}{}E[\widetilde{S}^x_t] =
(\by{1}{2}\Delta-\by{1}{2}\nabla\log\delta)E[\widetilde{S}^x_t]$
and
$\psi(t,x) := \delta^{-\frac{1}{2}}(x)E[\widetilde{S}^x_t]$
satisfies   
\begin{equation}\label{e:H2}
 \dd{t}{}\psi = \big(\by{1}{2}\Delta -
 \delta^{-\frac{1}{2}}\Delta\delta^{\frac{1}{2}}\big)\psi
\end{equation}
which is, of course, \eqref{e:H1} with $\lam=0$.  

If $f$ is an eigenfunction of $\Delta - \nabla\log\delta$ with
$(\Delta - \nabla\log\delta)f = \gamma f$
then it
follows that 
\[
 \psi(t,x)
 =
 \delta^{-\frac{1}{2}}(x) E[\widetilde{S}^x_t] = e^{\frac{1}{2}\gamma t}\delta^{-\frac{1}{2}}(x)f(x)
\]
is a stationary solution of \eqref{e:H2}.

To change the sign in the exponent of the measure term in \eqref{e:H2}
we alter the definition of the projected stochastic action
\eqref{e:pr-ac}: With $L_f$ as above we define
\begin{align}\label{e:pr-ac2}
 S^+(x,t,\om) 
 &:=
 f(\tau\circ(\Gamma^{\mathcal{O}})^{(x,\nabla f(x))}_t(\om))\\
 &\phantom{==}\notag
 +
 \int_0^t\big(i_{X_{h^0}}\theta-h^0\big)((\Gamma^{\mathcal{O}})^{(x,\nabla
   f(x))}_s(\om)) \;ds 
 +
 \sum\int_0^t\big(i_{X_{h^i}}\theta-h^i\big)((\Gamma^{\mathcal{O}})^{(x,\nabla
   f(x))}_s(\om)) \;dB_s^i . 
\end{align} 
This would be the \emph{time-forward projected stochastic action}. 
With the same reasoning as above it follows that
\begin{align}\label{e:ES+}
 E[S^+(x,t)] 
 &= E[f\circ X^x_t]
 = f(x) + E\left[\int_0^t\big(\by{1}{2}\Delta f 
   + \by{1}{2}\vv<\nabla f,\nabla\log\delta>\big)(X^x_s)\;ds\right].
\end{align}

\begin{proposition}
The function
\[
 \psi^+(x,t) := \sqrt{\delta(x)} E[S^+(x,t)] 
\]
satisfies 
$\psi^+(x,0) = \sqrt{\delta(x)}f(x)$ and
\begin{equation}\label{e:H3}
 \dd{t}{}\psi^+
 =
 \by{1}{2}\big(\Delta -
 \delta^{-\frac{1}{2}}\Delta\delta^{\frac{1}{2}}\big)\psi^+.
\end{equation}
Furthermore, if $f$ is an eigenfunction of $\Delta^0 =
\Delta+\nabla\log\delta$ with eigenvalue $\gamma$ then $\psi^+$ is a
stationary solution of \eqref{e:H3}, i.e.
\[
 \psi^+(t,x)
 = e^{\frac{1}{2}\gamma t}\sqrt{\delta(x)}f(x).
\]
\end{proposition}

\begin{proof}
Follows directly from the Kolmogorov backward equation applied to \eqref{e:ES+}.
\end{proof}

For $f=1$ and $\gamma=0$ we recover \eqref{e:H}.\footnote{Stationary
  solutions of \eqref{e:H3} correspond to stationary solutions of the
  corresponding Schr\"odinger equation $i\frac{\del}{\del t}\phi = -\by{1}{2}\big(\Delta -
 \delta^{-\frac{1}{2}}\Delta\delta^{\frac{1}{2}}\big)\phi$.}

\section{Reduction of Brownian
  motion and Quantum Calogero-Moser models}\label{sec:CMS}

\subsection{The Cartan decomposition}
Let $G$ be a semisimple Lie group with Lie algebra $\gu$ and Killing
form $B$. Consider a Cartan decomposition $\gu = \ko\oplus\po$ 
associated to the Cartan involution $\theta$,
and let
$G=K.\exp(\po)\cong K\times\po$, $g=k\exp x \leftmapsto(k,x)$ be the
corresponding decomposition of the group. Thus:
\[
[\ko,\ko]\subset\ko,\qquad
[\ko,\po]\subset\po,\qquad
[\po,\po]\subset\ko.
\]
Fix a maximal abelian subspace $\ao\subset\po$, and put
$\mo=Z_{\ko}(\ao)$ and $M=Z_K(\ao)$. 

Let $\Sigma$ be the set of restricted roots associated to the pair
$(\gu,\ao)$ and $\Sigma_+\subset \Sigma$ a choice of positive
roots. 
Then the associated root space decomposition is 
\[
 \gu
 =
 \gu_0\oplus\oplus_{\lam\in\Sigma}\gu_{\lam}
 \textup{ where }
 \gu_{0} = \mo\oplus\ao.
\]
We equip $\gu=\ko\oplus\po$ with the $\Ad$-invariant 
direct sum inner product
$\vv<.,.> = -B(.,\theta.) = -B|(\ko\times\ko) + B|(\po\times\po)$ 
and denote its restriction to $\ko$ or $\po$ by $\vv<.,.>$ again.

\subsection{The classical rational Calogero-Moser system associated to a
  Cartan decomposition}
Let $Q\subset\po$ denote the set of points which are regular with
respect to the $K$-action. Thus all isotropy subgroups $K_q$ are
conjugate to $M$ within $K$
and the quotient $Q/K=C$ is an open Weyl chamber.  
The Lagrangian $L$ is the kinetic energy
function 
\[ 
 L: TQ=Q\times\po\longto\R,
 \quad
 (q,v)\longmapsto\by{1}{2}\vv<v,v>=\by{1}{2}\mu(v,v)
\]
which is clearly $K$-invariant with respect to the tangent 
lifted $K$-action. 
Identifying $TQ=T^*Q$ via 
the metric we can equip $TQ$ with the standard symplectic
form $\Om$ and view $L$ as a Hamiltonian function $H$. 
Associated to the $K$-action there is now a \momap 
\[
 J: TQ\longto\ko^*=_{\vv<.,.>}\ko,
 \quad 
 (q,v)\longmapsto\ad(q).v=[q,v].
\]
Consider a (co-)adjoint orbit $\orb=\Ad(K).\lam\subset\ko$. Then the
Hamiltonian reduction of $(TQ,\Om,H)$ at the level $\orb$ yields a
reduced phase space of the form 
\begin{equation}\label{e:red-phase-space}
 TQ\sporb K 
 \cong
 TC\times\orb\spr{0}M
 \textup{ where }
 \orb\spr{0}M = (\orb\cap\mo^{\bot})/M.
\end{equation}
This isomorphism can be realized in terms of the mechanical
connection defined in \eqref{e:MC}.
The space $(\orb\cap\mo^{\bot})/M$ is, generally, not a manifold but
rather a stratified space.
Accordingly $TQ\sporb K$ is a stratified space, and each stratum
$TC\times(\textup{stratum})$ is equipped with a product symplectic
form such that the first factor is canonical while the second is
inherited from the KKS-form on $\orb$.
The reduced Hamiltonian becomes, in this picture, the Calogero-Moser
Hamiltonian function
\[
 H_{\textup{CM}}(x,v_0,[\lam])
 =
 \by{1}{2}\vv<v_0,v_0> 
 +
 \by{1}{2}\sum_{\alpha\in\Sigma_+}\by{\vv<\lam,\lam>}{\alpha(x)^2}
\] 
Note that the potential term
$\by{1}{2}\sum_{\alpha\in\Sigma_+}\by{\vv<\lam,\lam>}{\alpha(x)^2}
= \by{1}{2}\vv<A_x^*\lam,A_x^*\lam>
= \by{1}{2}\vv<\lam,\I_x^{-1}(\lam)>$ can be seen as an $M$-invariant function
on $C\times(\orb\cap\mo^{\bot})$.
See \cite{AKLM03,H04,FP06}.

\subsection{Stochastic reduction and Quantum CM systems}
As in Section~\ref{sub:zero-orb} we set $\orb=0$ in 
\eqref{e:red-phase-space} and hence $TQ\sporb K = TQ\spr{0}K = T(Q/K)
= TC \cong\R^{2l}$. 
Let $f: C\to\R$ be smooth and consider the forward projected
stochastic action 
\[
 S^+(x,t) = f\circ X^x_t
\]
from \eqref{e:pr-ac2} associated to the Lagrange submanifold $L_f =
\set{(x,\nabla f(x))}\subset TC$. 
As shown in Section~\ref{sub:zero-orb}, if $f$ is an eigenfunction of
$\Delta^0 = \Delta+\nabla\log\delta$ it follows that 
\[
 \psi^+(t,x) = \sqrt{\delta(x)}E[S^+(x,t)] 
 = e^{\frac{1}{2}\gamma t}\sqrt{\delta(x)}f(x)
\]
is a stationary solution to $\dd{t}{} - \by{1}{2}(\Delta -
\delta^{-\frac{1}{2}}\Delta\delta^{\frac{1}{2}}) = 0$. 
Such functions are provided by the zonal spherical functions. See
\cite[Appendix~C]{OP83}.
According to \cite[Theorem~5.1]{OP83} we have 
\[
 \by{1}{2}\delta^{-\frac{1}{2}}\Delta\delta^{\frac{1}{2}}
 =
 \sum_{\alpha\in\Sigma_+}\by{m_{\alpha}(2m_{\alpha}-2)|\alpha|^{2}}{8\alpha(x)^2}.
\]
Thus 
\[
 \psi(t,x) = \psi^+(it,x)
 = e^{i\frac{1}{2}\gamma t}\sqrt{\delta(x)}f(x)
\]
is a stationary solution to the Schr\"odinger equation
\begin{equation}\label{e:QCMS}
 i\dd{t}{}\psi
 =
 \left(
  -\by{1}{2}\Delta +
  \sum_{\alpha\in\Sigma_+}\by{m_{\alpha}(2m_{\alpha}-2)|\alpha|^{2}}{8\alpha(x)^2}
 \right)\psi
\end{equation}
where $m_{\alpha} = \dim\gu_{\alpha}$ is the multiplicity of the root. 
Equation~\eqref{e:QCMS} is the rational quantum Calogero-Moser
Schr\"odinger equation related to root systems of semi-simple Lie
algebras as described by  Olshanetsky and Perelomov~\cite{OP78,OP83}.

It will be a matter of future investigation to consider reduction at a
non-zero orbit $\orb$ and include a spin dependent potential in
\eqref{e:QCMS} as in \eqref{e:H1}.

\section{Appendix: Proof of Theorem~\ref{thm:e-r}}

Theorem~\ref{thm:e-r} is a result of \cite{HR10}. However, as this
paper is presently not published, we follow the referees' suggestion
and include a (slightly shortened)  proof for sake of completeness.

\begin{proof}[Proof of Theorem~\ref{thm:e-r}]
Let us begin by noting that $g\Gamma^{x,W}=\Gamma^{gx,\rho(g)W}$. Indeed,
\[
 \delta(g\Gamma^{x,W})
 = g\mathcal{S}(Y,\Gamma^{x,W})\delta Y
 = \mathcal{S}(\rho(g)Y,g\Gamma^{x,W})\delta(\rho(g)Y)
\]
whence $\tilde{\Gamma}:=g\Gamma^{x,W}$ satisfies
$\tilde{\Gamma}_0=gx$ a.s.\ and
$\delta\tilde{\Gamma}=\mathcal{S}(\rho(g)Y,\tilde{\Gamma})\delta(\rho(g)Y)$. By
existence and uniqueness of solutions the claim follows. 
In particular, we have 
$\pi\circ\Gamma^{x,W} = \pi\circ\Gamma^{gx,\rho(g)W}$. 

Claim:
\begin{equation}\label{e:claim1a}
 P_{gx} = g_*P_x 
\end{equation}
where $G$ acts on $W(Q)$ as $g: w\mapsto(t\mapsto gw(t))$. 
To see this let $S\subset W(Q)$ be a Borel cylinder set. This means
that there are $l\in\mathbb{N}$, $0\le t_1<\ldots<t_l\in
\mathbb{R}_+$, and a Borel set
$A\subset\Pi^l\dot{Q}$ such that 
$S = \ev(t_1,\dots,t_l)^{-1}(A)$,
where $\ev(t_1,\dots,t_l): W(Q)\to\Pi^l\dot{Q}$,
$w\mapsto(w(t_i))_{i=1}^l$. From the identity
$(\Gamma^{x,\rho(g)W})^{\check{}}_*P = (\Gamma^{x,W})^{\check{}}_*P$
we find
\begin{align*}
 P_{gx}(S)
 &= (\Gamma^{gx,\rho(g)W})^{\check{}}_*P(S)
  = P\set{\om: (\Gamma^{gx,\rho(g)W}_{t_i}(\om))_{i=1}^l \in A}\\
 &= P_x(\ev(t_1,\dots,t_l)^{-1}(g^{-1}A))
  = P_x(g^{-1}S).
\end{align*}

Consider the push forward map $\pi_*: W(Q)\to W(Q/G)$,
$w\mapsto\pi\circ w$. It is straightforward to see that
$\mathcal{B}(W(Q/G)) = \pi_*\mathcal{B}(W(Q))$. For $S_0 =
\pi_*(S)\in\mathcal{B}(W(Q/G))$ we may write the law 
$\left(P_{[x]}\right)_{[x]\in\dot{Q}/G}$ of $\pi\circ\Gamma$ as
\[
 P_{[x]}(S_0)
 = (\pi\circ\Gamma^{gx,\rho(g)W})^{\check{}}_*P(S_0)
 = P_{gx}(\pi_*^{-1}(S_0)).
\]
By \eqref{e:claim1a} this does not depend on $g\in G$.

Now, since $P_{[x]}$ is the push-forward of $P_x$ via $p$, we can use
the strong Markov property of $(P_x)_x$ to conclude that
$(P_{[x]})_{[x]\in\dot{Q}/G}$ satisfies the strong Markov property. 

To show that $\sum X_iX_i f\in\cinf(Q)^G$ for all $f\in\cinf(Q)^G$ consider
the standard basis $\{e_0,e_1,\dots,e_k\}$ of $\mathbb{R}\times\mathbb{R}^k$. For
$j=1,\dots,k$ we find
\[
 g\cdot X_j(x)
 =
 g\cdot \mathcal{S}(x,y,e_j)
 =
 \mathcal{S}(gx,\rho(g)y,\rho(g)e_j)
 =
 \sum_k g_{kj}X_k(gx),
\]
where $g_{kj} := \vv<e_k,\rho(g)e_j>$ is independent of 
$x\in Q$. Since $\sum_j g_{ij}g_{kj} = \delta_{ik}$,
\[
 X_i(gx) 
 = \sum_{j,k} g_{ij}g_{kj}X_k(gx)
 = \sum_j g_{ij}g\cdot X_j(x).
\] 
Thus $\Big(df(X_i)\Big)(gx)=\sum_j g_{ij}\Big(df(X_j)\Big)(x)$ for 
$f\in\cinf(Q)^G$ and also 
\[
 d\Big(df(X_i)\Big)(gx)\circ T_xg
 = d\Big(\sum_j g_{ij}df(X_j)\Big)(x)
 = \sum_j g_{ij}d\Big(df(X_j)\Big)(x).
\]
This implies that 
\begin{align*}
 \sum_i\Big(X_iX_if\Big)(gx)
 &= \sum_i\left\langle d \Big(df(X_i)\Big)(gx), X_i(gx)\right\rangle\\
 %&= \sum_id\Big(df\cdot X_i\Big)(gx)\cdot X_i(gx)\\ 
 &= \sum_{i,j,k} 
    \left\langle 
      g_{ij}d\Big(df(X_j)\Big)(x)\circ (T_xg)^{-1} , g_{ik}(T_xg)\cdot X_k(x)
    \right\rangle
  = \sum_i\Big(X_iX_if\Big)(x).
\end{align*}
Similarly, it is also easy to see that $X_0$ is $G$-invariant. Thus the
generator $A = X_0+\by{1}{2}\sum X_iX_i$ acts on $\cinf(Q)^G$,
whence it induces a projected operator $A_0$ characterized by
$A\circ\pi^* = \pi^*\circ A_0$.

Finally, 
to see that $A_0$ is the generator of $\pi\circ\Gamma$
we need to show that, for all $t\in\mathbb{R}_+$, $[x]\in Q/G$, 
and $f\in\cinf(Q/G)_0$, the $\mathbb{R}$-valued process 
\begin{align*}
 & M_t^f: \W(Q/G)\longto\mathbb{R},\\
 & M_t^f(w) 
   := f(w(t))-f(w(0))-\int_0^t(A_0f)(w(s))\,ds
\end{align*}
is a $P_{[x]}$-martingale on $(W(Q/G),\mathcal{B}(W(Q/G))$ for the filtration
$(\mathcal{B}_t(W(Q/G)))_t$. 
See \cite[Def.~IV.5.3]{IW89}.
This means that 
for all $t\ge0$,  $s\in[0,t]$, and $A\in\mathcal{B}_s(W(Q/G))$
we should check that
\[
 \int_A E^{P_{[x]}}\Big[M_t^f \Big| \mathcal{B}_s(W(Q/G))\Big](w)\,P_{[x]}(dw)
 = \int_A M_s^f(w)\,P_{[x]}(dw);
\]
see \cite[Chapter~V]{Str}.
Indeed,
\begin{align*}
 \int_A E^{P_{[x]}}\Big[M_t^f \Big| \mathcal{B}_s(W(Q/G))\Big](w)\,P_{[x]}(dw)
 &=  
 \int_A M_t^f(w)\,P_{[x]}(dw)
  = 
 \int_{p^{-1}A}(p^*M_t^f)(u)\,P_x(du)\\
% &= 
% \int_{p^{-1}A}(\hat{M}_t^{\pi^*f})(u)\,P_x(du)\\
% &=
% \int_{p^{-1}A} E^{P_x}\Big[\hat{M}_t^{\pi^*f} \Big| \mathcal{B}_s(W(Q))\Big](u)\,P_x(du)\\
% &=
% \int_{p^{-1}A} \hat{M}_s^{\pi^*f}(u)\,P_x(du)\\
 &=
 \int_{p^{-1}A}(p^*M_s^f)(u)\,P_x(du)
  =
 \int_A M_s^f(w)\,P_{[x]}(dw).
\end{align*}
%Here, $\hat{M}_t^{\pi^*f}: W(Q)\to\mathbb{R}$ is analogously 
%defined to $M_t^{f}$. We have used that $\hat{M}_t^{\pi^*f}$ 
%is a $P_x$-martingale with respect to $(\mathcal{B}_t(W(Q)))_t$ 
%for all $x\in Q$ and that $p^*M_t^f = \hat{M}_t^{\pi^*f}$ 
%which holds because of $(A_0f)\circ\pi= A\pi^*f$.
\end{proof}

\end{document}